\documentclass[preprint,12pt]{elsarticle}

\usepackage{lineno,hyperref}
\modulolinenumbers[5]

\usepackage[utf8]{inputenc}
\usepackage[utf8]{inputenc}
\usepackage[margin=1in]{geometry}
\usepackage{amsmath,amssymb,amsthm,graphicx,amsxtra, setspace}
\usepackage[utf8]{inputenc}
\usepackage{mathrsfs}
\usepackage{hyperref}
\usepackage{upgreek}
\usepackage{mathtools}
\usepackage{xcolor}
\usepackage{todonotes}
\allowdisplaybreaks

\newtheorem{theorem}{Theorem}[section]
\newtheorem{lemma}[theorem]{Lemma}
\newtheorem{proposition}[theorem]{Proposition}

\newtheorem{corollary}[theorem]{Corollary}
\newtheorem{definition}[theorem]{Definition}

\newtheorem{remark}[theorem]{Remark}

\journal{Journal of Mathematical Analysis and Applications \quad }

\begin{document}

\begin{frontmatter}



\title{Convergences for a Virus-like Evolving Population driven by Mutually-exciting Hawkes Processes} 


\author{Rahul Roy$^a$\footnote{Email: rahul.roy@iiitd.ac.in}} 

\affiliation{organization={Department of Mathematics, Indraprastha Institute of Information Technology},
            addressline={Shyam Nagar, Okhla Industrial Estate}, 
            city={New Delhi},
            postcode={110020}, 
            country={India}}

\author{Dharmaraja Selvamuthu$^b$\footnote{Email: dharmar@maths.iitd.ac.in}} 
\affiliation{organization={Department of Mathematics, IIT-Delhi, Indian Institute of Technology Delhi},
            addressline={Hauz Khas}, 
            city={New Delhi},
            postcode={110016}, 
            country={India}}

\author{Paola Tardelli$^c$\footnote{Email: paola.tardelli@univaq.it, CORRESPONDING AUTHOR}} 
\affiliation{organization={DIIIE, Department of Industrial and Information Engineering and Economics, University of L'Aquila},
            addressline={Piazzale E. Pontieri, 1}, 
            city={Roio},
            postcode={67100}, 
            country={Italy}}

\begin{abstract}
This paper presents a stochastic model motivated by the study of a virus-like evolving population with different mutation rates. 
This is a continuous time birth-death model: the birth processes are mutually-exciting  Hawkes processes and the death process is also a Hawkes process. 
This structure for the births and the deaths does not allow, in general, to get the Markov property of the processes involved. 
But considering the couple given by the Hawkes processes and their intensities we are able to deduce the necessary and sufficient conditions for the Markov property of the couple. 
This property is the main tool to get the convergence results describing the behaviour of the population, and the existence of a phase transition at a critical fitness level.
\end{abstract}



\begin{keyword}
Hawkes Process \sep Birth-Death Process \sep Dynamic Contagion Processes \sep Convergences \sep Virus Population

\MSC 60G55, 60J80, 60G52, 60K25, 60K30, 60J76
\end{keyword}

\end{frontmatter}



\section{Introduction}
\label{intro}

Inspired by the discrete time evolution model studied in 
\citet{RoyTanemura}, 
in the present paper, we develop a continuous time stochastic  model for the evolution and survival of species subject to death and  birth 
of individuals belonging to different mutation levels. 
The model initially introduced by 
\citet{GMS2011} consists of either a birth or a death taking place independently at each discrete time point with probability $p$ or $1-p$, respectively. Each individual at birth is assigned a fitness which is chosen uniformly at random from $[0,1]$ and is independent of all other processes. When a death occurs the individual with the smallest fitness is eliminated, and in case there is no individual present in the population then a death does not affect the system.  
In the modification of this model, 
\citet{BS2016} 
 and 
 \citet{RoyTanemura} assume that
\begin{itemize}
\item[(a)] a birth or a death takes place at a discrete time point with probability $p$ or $1-p$, respectively; at birth an individual is either a mutant with probability $r$ or an individual of one of the existing subspecies with probability $1-r$.
\item[(b)] In case the individual at birth is a mutant, then it is assigned a value, chosen uniformly at random from $[0,1]$, corresponding to the so-called fitness level.
\item[(c)] In case the individual belongs to one of the existing subspecies, 
in the 
\citet{BS2016} 
 model, the individual is assigned a fitness uniformly chosen from the existing fitness levels, 
while, in the 
\citet{RoyTanemura} model, the individual is assigned a fitness $f$ with a probability proportional to the number of individuals with fitness $f$ present in the  population.
\end{itemize}

Note that, even if, in several papers, the authors assume that 
the births and deaths are driven by Poisson processes, in many real-world problems, these assumptions are not valid. Often,
there are situations in which the intensity function for the births cannot be considered as a constant and, moreover,  
the rate of new birth events increases with each event occurs, and 
also new death  excites the further death process, see 
\citet{EMW}, 
\citet{DP}, 
\citet{KSBM}. 
Hawkes processes have been studied extensively in connection with modelling earthquakes, societal aberrations and social media trend.
In these problems,  
the future evolution of the models is influenced by the past events. 
The aforementioned empirical observations  motivate the development of models in the financial market, capable of incorporating some kind of contagion effect, see 
\citet{ait2015}.
Similarly, 
many authors have proposed various stochastic models using Hawkes processes for the temporal growth and migration of COVID-19 to incorporate contagion effect in the disease infection (see  
\citet{GLT2021} and references therein). 

This is the reason why, in the present paper, to incorporate the Darwinian theory of the survival of the fittest, we assume that the births occur according to mutually-exciting  Hawkes processes 
while the deaths occur according to a self-exciting  Hawkes process. 
Unlike the Poisson process, a Hawkes process is non-Markovian, in general, because the arrival of a particle excites the process such that  the chances of the next arrival increases for some time period subsequent to the arrival. 
An important property of Hawkes process, the self-exciting property, is well known for its efficiency in reproducing jump clusters due to the dependence of the conditional birth rate, or intensity, on the counting process itself. 
These processes have also recently been used in problems related to financial markets, 
for instance, 
\citet{ELL2011} discusses the multivariate Hawkes process and their applications to financial data. 



The contribution of this work is twofold: on one hand, the necessary and sufficient conditions are derived for the Markov property of the couple given by general mutually-exciting Hawkes processes and their intensities. These conditions apply, in particular, to the continuous time model  given by Hawkes births, with mutant and non-mutant individuals, respectively, $N^1$ and $N^2$, and  Hawkes deaths $N^3$, and their intensities. 
The specificity of the Hawkes processes $N^1$, $N^2$,  and $N^3$, and the queue process $N= N^1+N^2-N^3$ is the dependence of the intensity, on the entire size of the population, $N$, and not only on the size of $N^3$. 
On the other hand, some convergence results for the size of the population are reached, 
identifying the sites where the population is concentrated.

\section{The population model} 
\label{model}

We consider a population of individuals having different fitness levels: an individual may have a fitness which already exists in the population, or it is a mutant. 
The fitness level for each individual is a parameter  $f$, taking values uniformly in $[0, 1]$ and independently of other individuals. 
When the death process rings, an individual who has its fitness closest to the value $0$ dies.
 
This population dynamics is described by three stochastic processes: $N^1$ counting the births of new mutants, $N^2$ counting the births of non-mutants and $N^3$ counting the deaths of individuals. 
Each process $N^i$, $i=1,2,3$, is characterized by its own sequence of jump times, 
we consider $\{\tau_k\}_{k \geq 0}$ the sequence of jump times obtained by merging these three sequences of jump times.
By construction, we assume that initially the population is empty, therefore, 
\begin{itemize}
\item 
$\tau_0 = 0$, there is no individual at first.
\item 
$\tau_1$  is a jump time of the process $N^1$, since there can be neither a birth of a non-mutant, having a pre-existing fitness level  in the population, nor a death.
\item 
The following jump times can be a jump time from any process. The dynamics is given as below:
\begin{itemize}
\item 
If $\tau_k$ is a jump time of $N^1$, this means that a new mutant is born, and its fitness level is chosen uniformly at random in $[0, 1]$.
\item If $\tau_k$ is a jump time of $N^2$, this means that a non-mutant is born, and its fitness level is chosen from the set of all fitness levels already existing in the living population, with a probability proportional to the number of individuals with this fitness level.
\item
If $\tau_k$  is a jump time of $N^3$, the smallest fitness level is chosen so that an individual with this fitness level dies.
\end{itemize}
\end{itemize}

Inspired by the previous considerations, this article adopts that the dynamics of the birth processes are  
mutually-exciting Hawkes processes, while the death process is governed by a self-exciting Hawkes process. 
The uncertainty is assumed to be describe by 
$(\Omega, \mathcal{F},  \{\mathcal{F}_{t}\}_{t\geq0} ,P)$, a filtered probability space, 
where the filtration $\mathcal{F}_t$
satisfies the usual conditions of completeness and right continuity, and on which all the 
stochastic processes are defined.

\begin{definition}
\label{Hawkes-proc}
A {\it self-exciting Hawkes process} is a point process 
$H:=\{H(t): t \geq 0\}$ on $[0, \infty)$ 
and  ${\cal F}^H_t : = \sigma\left\{ H(s) : s \leq t \right\}$ is the filtration generated by $H$. 
In the linear case the intensity of  $H$ is given by
$\lambda(t) = \lambda_0 + \int_0^t \varphi(t-u) dH(u)$, 
where $\lambda_0$ is a positive real constant called background intensity.
The excitation function, $\varphi:(0,\infty) \to [0, \infty)$, is  $L^1$-integrable, 
and, to avoid the explosion of the process, must be $\| \varphi \|_{L^1} < 1$. 
Note that, when $\varphi (s) = 0$ for all $s$, the Hawkes process reduces to a  homogeneous Poisson process with parameter $\lambda_0$.
\end{definition}

\begin{definition}
The processes $N^1:= \{N^1(t): t \geq 0\}$ and $N^2:=\{N^2(t): t \geq 0\}$, are {\it mutually-exciting Hawkes processes} 
if, given  the background intensities $\lambda^1_0$ and $\lambda^2_0$, and the mutually excitation functions, $\varphi_{ji}$, $i, j \in\{1, 2\}$,
their intensities at time $t$ are given  by 
\begin{equation}
\label{generallambdai}
\lambda^i(t) = \lambda^i_0 + \sum_{j=1,2} \int_0^t \varphi_{ji}(t-u) dN^j(u), \qquad \text{ for } i = 1,2.
\end{equation}
\end{definition}

\begin{definition}
Let $N^3:=\{N^3(t): t \geq 0\}$ be another Hawkes process, independent of $N^1$ and $N^2$, 
with  background intensity $\lambda^3_0$ and excitation function $\psi$,  its intensity is given by
\begin{equation}
\label{lambda3}
\lambda^3(t) =  \lambda^3_0 + \int_0^t \psi(t-u) d N^3(u).
\end{equation} 
\end{definition}

\begin{remark}
We are considering a model which is a simple queueing process $N(t)$  adapted to $\mathcal{F}_t$ and with input process $N^1$ and $N^2$ and output driven by $N^3$. 
Hence, at time $t$, the size of the population is
\begin{equation}
\label{cardinality}
N(t):= N^1(t)+N^2(t)-N^3(t),
\end{equation} 
and only if $N(t)$ is positive there is the possibility of a death at time $t$. 
Therefore, 
\citet{Bremaud1981} (Chapter II Section 3), $N^3$ also admits the $\mathcal{F}_t$-intensity  $\lambda^3(t) {\bf 1}_{N(t-) > 0}$. 
By the linearity of the mean value and if $N^1$, $N^2$ and $N^3$ have the $\mathcal{F}_t$-intensities  $\lambda^1(t)$, $\lambda^2(t)$ 
and $\lambda^3(t) {\bf 1}_{N(t-) > 0}$, respectively, we are able to assert that $N$ has $\mathcal{F}_t$-parameters $\lambda^1(t)$, $\lambda^2(t)$ and 
$\lambda^3(t)$, when $\lambda^1(t)$ and $\lambda^2(t)$ are related to the birth intensities and $\lambda^3(t)$ is related to the death intensity.
\end{remark}


In  Section \ref{expNexplambda} we derive various properties of the Hawkes processes, like the mean values of all the intensities, the expected number of births and the expected number of deaths. 
Section \ref{representation} and Section \ref{MarkovSection}, are devoted to prove the Markov property of the couple given by the processes and their intensities. Section \ref{stability} investigates the stability of the model by studying the expectations of the intensities. Finally, Section \ref{Rahul-model} describes in details the model given by a virus-like evolving population and in Section \ref{model-convergences} there are some convergence results for the behaviour of the population, identifying the region where the population is concentrated. Specifically in Theorem \ref{T1-Rahul} and Corollary \ref{Cor2-Rahul} we prove that there exists a phase transition at a critical fitness level.

\section{Expected number of births and deaths and expected intensities}
\label{expNexplambda}

From all the assumptions made so far on 
(\ref{generallambdai}), 
(\ref{lambda3}) and 
(\ref{cardinality}), 
taking into account that  $\{\tau_i\}_{i \in \mathbb{N} \cup \{0\}}$ is the sequence of jump times of the process $(N^1, N^2, N^3)$, and 
setting as usual $\Delta N^j(t) = N^j(t) - N^j(t-)$, for $j=1,2,3$,  
the intensities admit the representations 
$$\lambda^i(t) = \lambda^i_0 + \sum_{j=1,2} \sum_{k=1}^\infty \varphi_{ji}(t-\tau_k) \Delta N^j(\tau_k) {\bf 1}_{\tau_k \leq t},  
\quad \mbox{for} \ i=1,2,$$
and 
$$\lambda^3(t) = \lambda^3_0 {\bf 1}_{N(t-) > 0}
+ \sum_{k=1}^\infty \psi(t - \tau_k) \Delta N^3(\tau_k) {\bf 1}_{\tau_k \leq t}
{\bf 1}_{N(t-) > 0}.$$ 

Note that, for $j=1,2,3$, $\Delta N^j(\tau_n)$ is not null only if $N^j$ has a jump in $\tau_n$. In this case, 
the intensity has an increment due to this jump. 
For instance, if $N^1$ has a jump at time $\tau_8$, 
then $\Delta N^1(\tau_8) = 1$ and $\Delta N^2(\tau_8) = 0$, which implies that
$\sum_{j=1,2} \varphi_{ji}(t-\tau_8) \Delta N^j(\tau_8) {\bf 1}_{\tau_8 \leq t} 
= \varphi_{1i}(t-\tau_8) {\bf 1}_{\tau_8 \leq t}.$
\begin{remark}
At  $t$, set 
$h := N^1(t) + N^2(t) + N^3(t)$, 
which is the number of jump times of 
$\left(N^1(t), N^2(t), N^3(t)\right)$.
Therefore,  the intensities 
for $\tau_h \leq t < \tau_{h+1}$ are given by
\begin{eqnarray}
\label{lambdaBbetween}
\lambda^i(t) &=& \lambda^i_0 + \sum_{j=1,2} \sum_{k=1}^h \varphi_{ji}(t-\tau_k) \Delta N^j(\tau_k), 
\qquad i=1,2,
\\
\label{lambdaDbetween}
\lambda^3(t) &=& \lambda^3_0 {\bf 1}_{N(\tau_h) > 0}
+ \sum_{k=1}^h \psi(t - \tau_k) \Delta N^3(\tau_k)
{\bf 1}_{N(\tau_h) > 0}.
\end{eqnarray}
Thus, at the jump time $\tau_h$ the intensities  are such that
\begin{eqnarray}
\label{lambdaBjump}
\lambda^i(\tau_h) - \lambda^i(\tau_h-) &=& \sum_{j=1,2} \varphi_{ji}(0) \Delta N^j(\tau_h), 
\qquad  i=1,2,
\\
\label{lambda3jump}
\lambda^3(\tau_h) - \lambda^3(\tau_h-) &=&  \psi(0) \Delta N^3(\tau_h).
\end{eqnarray}
Note that the sudden increase in $\tau_h$ implies a higher probability of having a jump of the same kind  immediately after $\tau_h$.
\end{remark}
The Hawkes processes are particular cases of non-homogeneous point processes. 
Their intensities are stochastic, since explicitly depend on the occurrence of previous events,
followed by a progressive normalisation, driven by the memory kernels.

\begin{proposition}
\label{systemODEBlambda}
For $i=1,2,3$, by 
(\ref{generallambdai}) and (\ref{lambda3}),  
$\left(\lambda^i(t), N^i(t) \right)$
are such that 
the expected intensities $\mathbb{E} [\lambda^i(t)]$ 
satisfy the differential equations:
\begin{eqnarray}
\label{ODExLambdaB}
\qquad d \mathbb{E} [ \lambda^i(t) ] 
&=& \sum_{j=1,2} \left( \int_0^t \varphi'_{ji}(t-u) \mathbb{E} [ \lambda^j(u)] du
+  \varphi_{ji}(0)  \mathbb{E} [ \lambda^j(t)] \right) dt, 
\qquad i=1,2
\\
\label{ODExLambdaD}
d \mathbb{E} [ \lambda^3(t) ] 
&=& \int_0^t \psi'(t-u) \mathbb{E} \left[ \lambda^3(u) \right] du \ dt 
+ \psi(0)  \mathbb{E} \left[ \lambda^3(t) \right] dt.
\end{eqnarray}
\end{proposition}

\begin{proof} 
For  $n > 0$, at $t = \tau_n$, a jump time of the process 
$(N^1, N^2, N^3)$, the intensities $\lambda^1$, $\lambda^2$, and $\lambda^3$ have a jump given by 
(\ref{lambdaBjump}) and (\ref{lambda3jump}), respectively, hence,
$$d \lambda^i(t) 
=  \sum_{j=1,2}  \varphi_{ji}(0)  d N^j(t), \quad \mbox{for} \ i =1,2,$$
and 
$$d \lambda^3(t) 
=   \psi(0)  d N^3(t).$$
In between the jump times, 
for $\tau_{n-1} \leq t < \tau_n$, given 
(\ref{lambdaBbetween}) and (\ref{lambdaDbetween}), we have that
$$d \lambda^i(t) = \sum_{j=1,2} \int_0^t \varphi'_{ji}(t-u) dN^j(u) \ dt,$$ 
and 
$$d \lambda^3(t) 
= \int_0^t \psi'(t-u) dN^3(u) \ dt.$$
This means that, for all $t$, 
$$d \lambda^i(t) = \sum_{j=1,2} \left\{  \int_0^t \varphi'_{ji}(t-u) dN^j(u) \ dt + \varphi_{ji}(0)  d N^j(t) \right\},$$ 
for $i=1,2$, and 
$$d \lambda^3(t) =  \int_0^t \psi'(t-u) dN^3(u) \ dt + \psi(0)  d N^3(t).$$

Recall that, (see the definition of intensity given by 
\citet{Bremaud1981}, Chapter II Section 3), the counting processes $N^i$ admits intensity $\lambda^i(t)$, if, for all non negative predictable process $C(t)$, the equality 
$\mathbb{E} [\int_0^\infty C(s) dN^i(s)] = \mathbb{E}[\int_0^\infty C(s) \lambda^i(s) ds]$ 
holds. 
Hence, by the Dominated Convergence Theorem, 
$\mathbb{E} [d \lambda^i(t)] = d \mathbb{E} [\lambda^i(t)]$, and 
$$d \mathbb{E} [\lambda^i(t)] = \sum_{j=1,2} 
    \left(  \int_0^t \varphi'_{ji}(t-u) \mathbb{E} [ \lambda^j(u) ]du \ dt 
    + \varphi_{ji}(0)  \mathbb{E} [ \lambda^j(t) ] dt \right),$$    
for $i=1,2$, and an analogous procedure allows us to deduce (\ref{ODExLambdaD}).
\end{proof}

\begin{theorem}
\label{EBElambda} 
Let 
$\left(\lambda^i(t), N^i(t)\right), i=1,2$, 
be Hawkes processes with dynamics of the intensities  given by 
(\ref{generallambdai}).  
Solving 
(\ref{ODExLambdaB}), 
for $i = 1, 2$, $j = 1, 2$ and $i \neq j$, we get
\begin{eqnarray}
\mathbb{E} [ \lambda^i(t) ] &=&  \lambda^i_0  \int_0^t \int_0^s
\left[ \varphi_{ij}(s-u) \varphi_{ji}(u)
-\varphi_{ii}(s-u) 
\varphi_{jj}(u) 
\right] \ du  \ ds
\nonumber\\
&&+ \lambda^i_0 
+ \lambda^i_0  \int_0^t \varphi_{ii}(s) \ ds 
+ \lambda^j_0  \int_0^t \varphi_{ji}(s) \ ds.
\nonumber
\end{eqnarray}
Furthermore, taking into account that $\mathbb{E} [ N^i(t)] = \int_0^t \mathbb{E} [ \lambda^i(v) ] \ dv$
\begin{eqnarray}
\mathbb{E} [ N^i(t)] 
&=& \lambda^i_0 \int_0^t \int_0^v \int_0^s
\left[ \varphi_{ij}(s-u) \varphi_{ji}(u)
-\varphi_{ii}(s-u) 
\varphi_{jj}(u) 
\right] \ du  \ ds \ dv
\nonumber\\
&&
+ \lambda^i_0 t
+ \lambda^i_0  \int_0^t \int_0^v \varphi_{ii}(s) \ ds \ dv
+ \lambda^j_0  \int_0^t \int_0^v \varphi_{ji}(s) \ ds \ dv.
\nonumber
\end{eqnarray}
\end{theorem}

The proof of Theorem \ref{EBElambda} is in Appendix. 

\begin{remark}
If  $\left(\lambda(t), N(t)\right)$ is  a linear  Hawkes process 
with exponential excitation function, namely, 
$\lambda(t) = \lambda_0 + \int_0^t \alpha e^{-\beta(t-u)} dN(u)$,
with $\alpha, \beta > 0$, 
(\ref{ODExLambdaB}) reduces to 
\begin{equation}
\label{ODExLambdaB-linearH}
d \mathbb{E} [ \lambda(t) ] 
= \left( \alpha  \mathbb{E} [ \lambda_t]  -\beta \int_0^t \alpha e^{-\beta (t-u)} \mathbb{E} [ \lambda(u)] du \right) dt.
\end{equation}
Taking the Laplace Transform on both sides, and setting as usual  
$Y(s) = {\cal L}(\mathbb{E} [ \lambda(t) ] )(s) = \int_0^\infty e^{-s t} \mathbb{E} [ \lambda(t) ]  dt$,
and rearranging, we get
$$Y(s) \left( 1 + (\beta - \alpha) \frac{1}{s} \right) = \lambda_0 \left(  \frac{1}{s}  + \beta  \frac{1}{s^2} \right).$$
Taking the Inverse Laplace Transform and then the derivative of both sides, we have
\begin{equation}
\label{ODExLambdaB-linearH-final}
d \mathbb{E} [ \lambda(t) ] + (\beta - \alpha) \mathbb{E} [ \lambda(t) ]  = \beta \lambda_0.
\end{equation}
Alternatively, noting that,  in this case, 
$\mathbb{E} [\lambda(t) ] = \lambda_0 + \int_0^t \alpha e^{-\beta(t-u)} \mathbb{E} [\lambda(u) ] du$, 
immediately, (\ref{ODExLambdaB-linearH}) reduces to (\ref{ODExLambdaB-linearH-final}). 
Then,  integrating, we get
$\mathbb{E} [ \lambda(t) ] = \frac{\beta \lambda_0}{\alpha - \beta} \left( e^{(\alpha - \beta)t} - 1 \right) 
+ \lambda_0 e^{(\alpha - \beta)t}$, 
and, given $N(0) = N_0$,
$\mathbb{E} [ N(t) ] = N_0 + \frac{\alpha \lambda_0}{(\alpha - \beta)^2} \left( e^{(\alpha - \beta)t} - 1 \right) 
- \frac{\beta \lambda_0}{\alpha - \beta} \ t$.

Moreover, note that Theorem \ref{EBElambda} applies even to the special case of Poisson scenario, 
in which all the excitation functions are null.
\end{remark}

\begin{theorem}
\label{Exlambda3}
Given the 
process 
$\left(\lambda^3(t), N^3(t)\right)$ with the structure of the intensity process  given by 
(\ref{lambda3}), 
and solving 
(\ref{ODExLambdaD}), we have
$\mathbb{E} [ \lambda^3(t) ] = \lambda^3_0 + \lambda^3_0 \int_0^t \psi(v) \ dv$ and 
$\mathbb{E} [ N^3(t) ] = \int_0^t \mathbb{E} [ \lambda^3(v) ] \ dv 
= \lambda^3_0 t 
+ \lambda^3_0 \int_0^t \int_0^s \psi(v) \ dv \ ds.$
\end{theorem}
The proof of Theorem \ref{Exlambda3} is in Appendix.

\begin{remark}
\label{general-stability}
Recall that, in queuing theory, \it{stable} describes a system where the number of customers or jobs in the queue remains finite and does not grow indefinitely over time, indicating a balance between arrival and service rates, 
which, in our model, requires that
$$\mathbb{E} [ \lambda^1(t) ] + \mathbb{E} [ \lambda^2(t) ] < \mathbb{E} [ \lambda^3(t) ],$$
see also Remark \ref{widesensestazionarity}  later.
Looking at Theorem \ref{EBElambda} and Theorem \ref{Exlambda3}, 
we get that
\begin{eqnarray}
&&\mathbb{E} [ \lambda^1(t) ] + \mathbb{E} [ \lambda^2(t) ] 
= \lambda^1_0 + \lambda^1_0  \int_0^t 
(\varphi_{11}(s) + \varphi_{12}(s) )  ds
+ \lambda^2_0 + \lambda^2_0  \int_0^t 
(\varphi_{21}(s) + \varphi_{22}(s) )  ds
\nonumber\\
&& \qquad \qquad + (\lambda^1_0 + \lambda^2_0)  
\int_0^t \int_0^s
\left[ \varphi_{12}(s-u) ) \varphi_{21}(u)
-\varphi_{11}(s-u) 
\varphi_{22}(u) 
\right]  du   ds,
\nonumber
\end{eqnarray}
and $\displaystyle \mathbb{E} [ \lambda^3(t) ] = \lambda^3_0 + \lambda^3_0 \int_0^t \psi(v) \ dv.$
Therefore, the stability condition is given by
\begin{eqnarray}
\label{stab-cond}
&&\qquad \lambda^3_0 + \lambda^3_0 \int_0^t \psi(v) \ dv 
>  \lambda^1_0  \int_0^t 
(\varphi_{11}(s) + \varphi_{12}(s) )  ds
+ \lambda^2_0  \int_0^t 
(\varphi_{21}(s) + \varphi_{22}(s) )  ds
\\
&&\qquad 
+ (\lambda^1_0 + \lambda^2_0)  \left( 1 + 
\int_0^t \int_0^s
\left[ \varphi_{12}(s-u)  \varphi_{21}(u)
-\varphi_{11}(s-u) 
\varphi_{22}(u) 
\right] \ du  \ ds \right).
\nonumber
\end{eqnarray}
Note that, when the Hawkes process reduces to a Poisson process, the excitation functions are null, and,
$\mathbb{E} [ \lambda^i(t) ] \equiv \lambda_0^i$. Thus, the stability condition turns to be, as usual, 
$\lambda^3_0 > \lambda^1_0 + \lambda^2_0$.
\end{remark}

\section{Representation for the intensities}
\label{representation}

By 
\citet{EGG2010}, we know that 
the Markov property holds for the couple, 
Hawkes process and its intensity, having  exponential kernel. 
We are going to discuss this approach  investigating in this section and in the next one 
the necessary and sufficient conditions for the  Markov property of our model. 

First of all, due to the presence of the birth processes, the death process and the process counting the number of individuals alive, we cannot be sure that if there is a martingale defined with respect to a filtration, still it remains a martingale with respect to a greater filtration. 
More specifically, when there are two filtrations ${\cal G}^1$ and 
${\cal G}^2$, such that ${\cal G}^1 \subset {\cal G}^2$, as well known, 
a ${\cal G}^2$-martingale implies a 
${\cal G}^1$-martingale. 
But the converse is not true, namely, a ${\cal G}^1$-martingale does not always remain a martingale with respect to the greater filtration ${\cal G}^2$.
This is the classical problem of the enlargement of filtrations.
A usual assumption in this frame is the so-called Immersion property, 
namely, any ${\cal G}^1$-local martingale must also be a ${\cal G}^2$-local martingale, see 
\citet{T2017}, 
\citet{DT2022}
and for further details 
\citet{JYC2009}. 
Consequently, following the existing literature, the Immersion property is assumed in the present paper. 
Note that this is an essential tool to prove the Markov property of the model.

For $i=1, 2, 3$, the behaviour of each $\lambda^i$ is completely determined by a function of the history of the processes $N^i$,  up to time $t$, namely, by the so-called  
{\it history process} defined by $H^i_t=\left\{ \phi^i = (t^i_k)_{k \in \mathbb{N}}: t^i_k \leq t \right\}$, for $t > 0$, and 
$H^i_0 = \ \ \emptyset$, otherwise. 
These  are pure jump processes having the same jump times of 
$N^i$ and  taking values in the space of the history sets  
${\cal H}^i = \{\emptyset\} \cup \{\cup_{j \geq 1} {\cal H}^i_j \}$, 
where
$${\cal H}^i_j = \left\{ h = \left(t^i_1, \ldots, t^i_j \right): 
\ 0 < t^i_k<t^i_{k+1}, \  k =1, 2, \ldots, j \right\}.$$
By defining a suitable metric $d$, $({\cal H}^i, d)$ is a complete and separable metric space, 
and ${\cal B}({\cal H}^i)$ are the Borel subsets of ${\cal H}^i$. 
Given a $h \in {\cal H}^i$, let $j$ be 
the number of elements of $h$, then $h \in {\cal H}^i_j$, and 
$\sum_{k=1}^j {\bf 1}_{t^i_k \leq t}$, the corresponding trajectory, is the realization of $N^i$.
Thus, $H^i_t$ is a Markov process. 
Basic properties on the history process are in 
\citet{GTFiltering},
and further details in 
\citet{AHN}.

Unfortunately, the history processes are not really easy to handle. 
Hence, we want to avoid the use of them, which necessitates the introduction of some  heavy technicalities, 
but, at the same time, we must consider that  the vector of processes  $(N^1, N^2, N^3)$ 
still depend on their own whole history. 
To overcome this difficulty, we introduce the {\it shot noise processes}, in 
Definition \ref{def-xi} below.

\begin{definition}
\label{def-xi}
For $t > 0$, let  the  \underline{shot noise processes} be $(\xi^1_t, \xi^2_t, \xi^3_t)$ 
with 
\begin{equation}
\label{xikt-xi3t}
\xi^i_t := \sum_{j=1,2} \int_0^t \varphi_{ji}(t-u) dN^j(u)
\quad i = 1, 2, 
\qquad \xi^3_t =  \int_0^t \psi(t - u) dN^3(u).
\end{equation}
\end{definition}
Thus, $(\xi^1, \xi^2, \xi^3)$ contains all the information relating to the history of 
$(N^1, N^2, N^3)$ and of interest for $\lambda^i$, $i=1,2,3$,
given in 
(\ref{generallambdai}) and (\ref{lambda3}).
Consequently, 
\begin{equation}
\label{lambdakt}
\lambda^i(t)  = \lambda^i_0 + \xi^i_t,
\qquad i = 1, 2, 
\qquad \mbox{and} \qquad \lambda^3(t)  = ( \lambda^3_0 + \xi^3_t) {\bf 1}_{N(t-) > 0},
\end{equation}
where, $N(t)$ is the cardinality of the population at time $t$, as given in 
(\ref{cardinality}). 
Note that, 
(\ref{lambdakt}) implies that, 
$\lambda^i(t)$ are  deterministic measurable functions of 
$(\xi^1, \xi^2, \xi^3)$ at $t$.

\begin{theorem}
\label{lambdastructure}
Fix a $t$ and set $n := N^1(t) + N^2(t) + N^3(t)$, at the jump time $\tau_n$, 
\begin{eqnarray}
\label{lambdaitaun}
\lambda^i(\tau_n) &=& \lambda^i(\tau_n-) + \sum_{j=1,2}  \varphi_{ji}(0) \Delta N^j(\tau_n),
\qquad i = 1, 2,
\\
\label{lambda3taun}
\lambda^3(\tau_n)  &=& \left(\lambda^3(\tau_n-) 
+ \psi(0) \Delta N^3(\tau_n)\right)
{\bf 1}_{N(\tau_n-) > 0},
\end{eqnarray}
while, for  $\tau_n \leq s < t  \leq \tau_{n+1}$ 
and $i = 1, 2$, 
$$\lambda^i(t)  =  \lambda^i_0 +  \xi^i_s + 
\sum_{j =1,2} \ \sum_{k=1}^n
\left( \varphi_{ji}(t - \tau_k) - \varphi_{ji}(s - \tau_k) \right) \Delta N^j(\tau_k),$$
$$\lambda^3(t)  =  
\left( 
\lambda^3_0 + \xi^3_s + \sum_{k=1}^n \left(\psi(t - \tau_k) - \psi(s - \tau_k) \right) \Delta N^3(\tau_k) \right)
{\bf 1}_{N(\tau_n) > 0}.$$
\end{theorem}

\begin{proof} 
Note that, at the jump time $\tau_n$, (\ref{lambdaitaun}) and (\ref{lambda3taun}) immediately follow by (\ref{lambdakt}), 
and  by (\ref{xikt-xi3t}). 
But, for $\tau_n \leq s < t < \tau_{n+1}$, since
\begin{eqnarray}
\xi^i_t -   \xi^i_s 
&=& 
\sum_{j =1,2} \ \sum_{k=1}^n
\left( \varphi_{ji}(t - \tau_k) - \varphi_{ji}(s - \tau_k) \right) \Delta N^j(\tau_k),
\qquad i = 1, 2,
\nonumber\\
\xi^3_t - \xi^3_s 
&=& \hspace{- 4 pt}
\int_0^s \hspace{- 4 pt}  \left(\psi(t - u) - \psi(s-u) \right) dN^3(u)
= \hspace{- 4 pt} \sum_{k=1}^n
\left( \psi(t - \tau_k) - \psi(s - \tau_k) \right) \Delta N^3(\tau_k), 
\nonumber
\end{eqnarray}
the representation for intensity processes is achieved as required. 
\end{proof}

\begin{corollary}
For all \ $0 < s < t$, and for $i = 1, 2$, by (\ref{xikt-xi3t}), and (\ref{lambdakt}), the 
representation for the intensity processes are given by
\begin{eqnarray}
\lambda^i(t) &=& \lambda^i_0 + \xi^i_s 
+ \sum_{j=1,2} \left( \int_0^s \left( \varphi_{ji}(t-u) - \varphi_{ji}(s-u) \right) dN^j(u) 
+ \int_s^t \varphi_{ji}(t-u)  dN^j(u)  \right),
\nonumber\\
\lambda^3(t) &=& 
\left( \lambda^3_0 + \xi^3_s 
+ \int_0^s \left(\psi(t - u) - \psi(s - u) \right) dN^3(u)  
+ \int_s^t \psi(t - u) dN^3(u)  
\right) {\bf 1}_{N(t-) > 0}.
\nonumber
\end{eqnarray}

\end{corollary}

\section{Discussion on the Markov property}
\label{MarkovSection}

For $i = 1, 2, 3$, the representations of the shot noise $\xi^i$ and of the intensity $\lambda^i$, obtained in Theorem \ref{lambdastructure}, are the key to deduce the assumptions necessary to guarantee the Markov property of this model.

\begin{definition}
\label{Markovianity}
For  \  $0 < s < t$,
let us assume that there exist
$f_i(s,t)$,  $i = 1, 2, 3$, deterministic functions, such that the quantities
$$\varphi_{ji}(t-\tau_k) - f_i(s,t) \varphi_{ji}(s-\tau_k)  \quad i, j = 1, 2,
\quad \mbox{and} \quad 
\psi (t - \tau_k) - f_3 (s,t) \psi (s - \tau_k),$$
do not dependent of the jump time $\tau_k$, for any $\tau_k < s$. Then, we define
\begin{eqnarray}
\gamma_j(i, s, t) &:=& \varphi_{ji}(t-\tau_k) - f_i(s,t) \varphi_{ji}(s-\tau_k)  \qquad i, j = 1, 2,
\nonumber\\
\gamma(3, s, t) &:=&  \psi (t - \tau_k) - f_3 (s,t) \psi (s - \tau_k).
\nonumber
\end{eqnarray}
\end{definition}

\begin{theorem}
Under Definition \ref{Markovianity},  $(N^1, \lambda^1, N^2, \lambda^2, N^3, \lambda^3)$, the Hawkes processes and their intensities, satisfy the Markov property.
\end{theorem}
\begin{proof}
Note that 
\begin{equation}
\label{numero1-i}
\xi^i_t - f_i(s,t) \xi^i_s = \sum_{j=1,2} \left( \gamma_j(i, s, t) N^j(s)  + \int_s^t \varphi_{ji}(t-u)  dN^j(u)  \right), \quad \mbox{for} \  i=1, 2,
\end{equation}
and 
\begin{equation}
\label{numero1-3}
\xi^3_t - f_3(s,t) \xi^3_s = \gamma(3, s, t) N^3(s)  + \int_s^t \psi(t - u) dN^3(u).
\end{equation}
Consequently, 
$$\lambda^i(t) = \lambda^i_0 + f_i(s,t) \xi^i_s +  \sum_{j=1,2} \left( \gamma_j(i, s, t) N^j(s) 
+ \int_s^t \varphi_{ji}(t-u)  dN^j(u) \right), \quad \mbox{for} \ i = 1, 2,$$
and 
$$\lambda^3(t) =
\left( 
\lambda^3_0 + f_3(s,t) \xi^3_s + \gamma(3, s, t) N^3(s) + \int_s^t \psi(t - u) dN^3(u)
\right)
{\bf 1}_{N(t-) > 0}.$$ 
Hence, for $i=1,2,3$, we are able  to represent $\xi^i$ and 
$\lambda^i$, at time $t$, in term of themselves at time $s$, $s < t$, and  of some functions of $s$ and $t$. 
This structure guarantees the claim.
\end{proof}

\begin{proposition}
\label{f-exp}
Under Definition \ref{Markovianity}, $f_i(s,t)$ must be necessarily an exponential function, 
namely, there exists a positive real constant $\beta_i$ such that
$f_i (s, t) = e^{-\beta_i (t - s)}.$
\end{proposition}
\begin{proof}
For $0 < s < t$, let $\tau_k$ and $\tau_l$ be jump times of $(N^1, N^2, N^3)$,   which are also jump times of $N^j$, $j=1,2$, such that $\tau_k \vee \tau_l < s$. 
Then, taking into account Definition \ref{Markovianity},
\begin{equation}
\label{numero2}
\gamma_j(i, s, t) 
= \varphi_{ji}(t-\tau_k) - f_i(s,t) \varphi_{ji}(s-\tau_k)
= \varphi_{ji}(t-\tau_l) - f_i(s,t) \varphi_{ji}(s-\tau_l),
\end{equation}
for $i=1,2$. 
Hence,
\[
f_i(s, t) = \frac{ \varphi_{ji}(t-\tau_k) - \varphi_{ji}(t-\tau_l) }{ \varphi_{ji}(s-\tau_k) - \varphi_{ji}(s-\tau_l) },
\]
and 
for each $r$ such that  $s \leq r \leq t$, $f_i(s, t) =  f_i(s, r) \cdot f_i(r, t)$.
Considering the sequence of times $\{\tau'_j\}$, where $\tau'_j = \tau_j -s +1$, then
\[
f_i(s, t) 
= \frac{ \varphi_{ji}(t-s + 1-\tau'_k) - \varphi_{ji}(t - s + 1 - \tau'_l) }
{ \varphi_{ji}(1-\tau'_k) - \varphi_{ji}(1-\tau'_l) }
= \tilde f_i(t-s),
\]
and setting $u:= t - r$ and $v := r - s$, we get the Cauchy functional identity 
$\tilde f_i (u+v) = \tilde f_i (u) \cdot \tilde f_i (v)$.
As well known, this implies that $\tilde f_i (0) = 1$ and 
$\tilde f_i (u) = e^{-\beta_i u}$, for a $\beta_i \in \mathbb{R}$.  
Analogous procedure allows us to achieve the same result for $f_3(s,t)$.
\end{proof}

We have to highlight the results of Theorem \ref{exp-kernel} below, since (\ref{Phik-linear-exponential}) establishes an important restriction on the family of memory kernels of Hawkes processes that satisfy the  Markov property.
\begin{theorem}
\label{exp-kernel}
The family of memory kernels satisfying Definition \ref{Markovianity} is 
\begin{equation}
\label{Phik-linear-exponential}
\varphi_{ji}(t) = \delta_{ji} + \alpha_{ji} e^{-\beta_i t}, \quad i,j = 1,2,
\qquad \mbox{and} \qquad 
\psi(t) = \delta_3 + \alpha_3 e^{-\beta_3 t},
\end{equation}
where $\alpha_{ji}$, $\delta_{ji}$, $i,j = 1,2$, $\beta_i$, $i = 1,2, 3$, $\alpha_3$, and $\delta_3$ 
are positive real constants.
\end{theorem}
\begin{proof}

Taking into account that, under Definition \ref{Markovianity}, $f_i (s,t) = e^{-\beta_i (t-s)}$, 
and choosing $\tau_k$ and $\tau_l$,  
jump times of 
$(N^1, N^2, N^3)$, 
by (\ref{numero2}), we have that 
$$\varphi_{ji}(t-\tau_k) - \varphi_{ji}(t-\tau_l) 
= e^{-\beta_i (t-s)} \left[ \varphi_{ji}(s-\tau_k) - \varphi_{ji}(s-\tau_l) \right].$$
Choosing again $\tau'_k = \tau_k -s +1$, 
$$\varphi_{ji}(t-s + 1 -\tau'_k) - \varphi_{ji}(t-s+1 -\tau'_l) 
= e^{-\beta_i (t-s)} \left[ \varphi_{ji}(1-\tau'_k) - \varphi_{ji}(1-\tau'_l) \right],$$
for all $k$, 
which allows us to deduce that there exist the constants 
$\alpha_{ji}$, $\beta_i$ and $\delta_{ji} \in \mathbb{R}$, such that $\varphi_{ji}(t)$ assumes the form given in 
(\ref{Phik-linear-exponential}). 
Analogously, we are able to get the constants $\alpha_3$, $\beta_3$ and $\delta_3 \in \mathbb{R}$.
\end{proof}

\begin{corollary}
Given $f_i (s, t) = e^{-\beta_i (t - s)}$ and the results achieved in Theorem \ref{exp-kernel}, immediately, we get that, for $i, j = 1,2$,
$\gamma_j(i, s, t) = \delta_{ji} \left[ 1 - e^{-\beta_i (t - s )}  \right]$ and 
$\gamma(3, s, t) = \delta_3 \left[ 1 - e^{-\beta_3 (t - s )}  \right]$.
\end{corollary}

Definition \ref{Markovianity} is the fundamental tool to get the Markov property. 
Proposition \ref{f-exp} and Theorem \ref{exp-kernel} assert that the Hawkes processes that satisfies Definition \ref{Markovianity} must have necessarily memory kernels with the structure 
(\ref{Phik-linear-exponential}). 
This allows us to assure that the class of memory kernels of Hawkes processes satisfying the Markov property cannot be larger than that having the structure given in  (\ref{Phik-linear-exponential}). 
An analogous procedure generalises this result to a finite number of Hawkes processes, Theorem \ref{CNES}.

\begin{theorem}
\label{CNES}
The mutually exciting Hawkes processes and their intensity processes 
$$\left( N^1, \lambda^1, N^2, \lambda^2, \ldots, N^m, \lambda^m  \right)$$
satisfy the Markov property if and only if 
the memory kernels assume the form given in (\ref{Phik-linear-exponential}).
\end{theorem}

\begin{remark}
\label{widesensestazionarity}
Note that from 
\citet{LimaChoi} and 
\citet{Hawkes1971}
on the assumption  that
$\int_0^{+\infty} \varphi_{ji}(t) \ dt \leq 1$ and $\int_0^{+\infty} \psi(t) \ dt \leq 1$, 
we are assured that the corresponding processes will attain wide-sense stationary behaviours. 
For the present model, these assumptions allow us to write that
$\mathbb{E} [ \lambda^i(t) ] < 4 (\lambda^1_0 \vee \lambda^2_0)$, for $i =1, 2$, and $\mathbb{E} [ \lambda^3(t) ] < 2 \lambda^3_0$. 
Note that, 
under Definition \ref{Markovianity}, namely in the Markov case, if we want non-explosion of the processes, $N^1$, $N^2$, and $N^3$,
we must assume  that, for $i,j =1,2$,
\begin{equation}
\label{ws-stat-behav-0}
\delta_{ij} = 0, 
\qquad \delta_3 = 0,
\end{equation}
\begin{equation}
\label{ws-stat-behav-<}
\qquad \alpha_{ji} \leq \beta_i \qquad \mbox{and} \qquad \alpha_3 \leq \beta_3.
\end{equation}
\end{remark}

\subsection{Dynamical behaviour of the intensities}

Under Definition \ref{Markovianity}, we are going to give the representations of the behaviours of $\lambda^i$ and $\xi^i$, for $i=1,2,3$. 
Then, we write the infinitesimal generator 
for $(N^1, \lambda^1, N^2, \lambda^2, N^3, \lambda^3)$, given by the mutually exciting Hawkes processes, the self-exciting  Hawkes process and their intensities.

\begin{proposition}
\label{differentiable}
Under Definition \ref{Markovianity}, 
for $i=1,2,3$, we have that
\[
d\lambda^i(t) = d\xi^i_t  = \beta_i  \left(
\sum_{j=1,2}  \delta_{ji} N^j(t-) - \lambda^i(t) + \lambda^i_0 \right) dt 
+  \sum_{j=1,2}  \alpha_{ji} d N^j(t), 
\quad i = 1, 2, 
\]
\[
d\lambda^3(t) = d\xi^3_t = \beta_3 \left( \delta_3 N^3(t-) 
- \lambda^3(t) + \lambda^3_0 
\right) dt 
+ \alpha_3 d N^3(t).
\]
\end{proposition}

\begin{proof}
Fix a $t > 0$ and a positive   $n = N^1(t) + N^2(t) + N^3(t)$.
For $h$ small enough such that $\tau_n \leq t < t + h < \tau_{n+1}$, 
by (\ref{numero1-i}) and  (\ref{numero1-3}), 
we get that
\[
d\lambda^i(t) = d\xi^i_t  
= \beta_i \left( \sum_{j=1,2}  \delta_{ji} N^j(t) - \xi^i_t \right) 
= \beta_i \left( \sum_{j=1,2}  \delta_{ji} N^j(t) - \lambda^i_t + \lambda^i_0 \right),
\quad i = 1, 2, 
\]
and 
\[
d\lambda^3(t) 
= d\xi^3_t  
= \beta_3 \left( \delta_3 N^3(t) - \xi^3_t \right)
= \beta_3 \left( \delta_3 N^3(t) 
- \lambda^3(t) + \lambda^3_0 \right).
\]
By (\ref{lambdaitaun}) and (\ref{lambda3taun}), we get the claim by noting that, for $i = 1, 2$,

$\displaystyle d\lambda^i(\tau_n) = d \xi^i_{\tau_n} =   \sum_{j=1,2}  \alpha_{ji} d N^j(\tau_n), \quad  \mbox{and} \quad   d\lambda^3(\tau_n) = d \xi^3_{\tau_n} = \alpha_3 d N^3(\tau_n).$
\end{proof}

\begin{theorem}
\label{generator}
Under Definition \ref{Markovianity}, let $\zeta:= \left(n_1, l_1, n_2, l_2, n_3, l_3 \right)$.
Let $\mu_i$ be the transition probability 
if an increment for $N^i$, $i=1,2, 3$,  takes place. 
For a suitable choice of 
$F(\zeta)$ (differentiable with respect to $l_i$), 
\citet{EK}, a representation for the 
generator $\cal G$ of the process 
$\left(N^1, \lambda^1, N^2, \lambda^2, N^3, \lambda^3 \right)$ is given by
\begin{eqnarray}
\label{Ggenerator}
{\cal G}F(\zeta) &=&  \sum_{i=1,2} \biggl\{ {\cal G}^i F(\zeta) +  l_i  \left[ F(\zeta + \Delta_i \zeta) - F(\zeta) \right] \mu_i(\zeta) \biggr\}
\\
&& \qquad \qquad 
+ {\cal G}^3 F(\zeta) +  l_3 {\bf 1}_{n_1 + n_2 - n_3 > 0} \left[ F(\zeta + \Delta_3 \zeta) - F(\zeta) \right] \mu_3(\zeta),
\nonumber
\end{eqnarray}
where, for $i = 1, 2$,
$${\cal G}^i F(\zeta)  
= \beta_i  \left(
\sum_{j=1,2}  \delta_{ji} n_j - l_i + \lambda^i_0 \right) \frac{\partial F(\zeta)}{\partial l_i},
\quad 
{\cal G}^3 F(\zeta)  = \beta_3 \left( \delta_3 n_3 - l_3 + \lambda^3_0 \right)
\frac{\partial F(\zeta)}{\partial l_3},$$
and 
\begin{equation}
\label{jumpsizes}
\Delta_1 \zeta = \left(1, \alpha_{11}, 0, \alpha_{12}, 0, 0 \right), 
\Delta_2 \zeta = \left(0, \alpha_{21}, 1, \alpha_{22}, 0, 0 \right), 
\Delta_3 \zeta = \left(0, 0, 0, 0, 1, \alpha_3 \right).
\end{equation}
Hence, 
$\left( N^1, \lambda^1, N^2,  \lambda^2, N^3, \lambda^3 \right)$ is a Markov process 
which solves the Martingale problem for $\cal G$, $MgP({\cal G})$.
\end{theorem}
The proof follows by standard procedure, see for instance Theorem 7.3 in 
\citet{EK}.

\begin{corollary}

Under Definition \ref{Markovianity} and the conditions given in Remark \ref{widesensestazionarity}, the representation for the 
generator $\cal G$ of the process 
$\left(N^1, \lambda^1, N^2, \lambda^2, N^3, \lambda^3 \right)$ given in Theorem \ref{generator} reduces to 
(\ref{Ggenerator}),  where
${\cal G}^i F(\zeta)  
= \beta_i  \left(
 \lambda^i_0 - l^i \right) \frac{\partial F(\zeta)}{\partial l^i}$, $i = 1, 2, 3$,
and the sizes of the jumps are given by 
(\ref{jumpsizes}). 
Hence, again, 
$\left( N^1, \lambda^1, N^2,  \lambda^2, N^3, \lambda^3 \right)$ is a Markov process 
which solves the Martingale problem for $\cal G$, $MgP({\cal G})$.

\end{corollary}

\section{Stability}
\label{stability}
Recall that, (\ref{cardinality}),
the model studied is a simple queueing process $N(t)$, and its size is given by 
$N(t) = N^1(t) + N^2(t)  - N^3(t)$, a process 
having  intensity $\lambda^1(t) + \lambda^2(t) - \lambda^3(t)$, 
and taking non-negative integer values. This process $\{N(t) : t \geq 0\}$ is the continuous time interpolation 
of the discrete time process $\{N(\tau_n) : n \geq 0\}$. 
Before giving the proof of Theorem \ref{T1-Rahul}, some preliminaries are needed.
\begin{lemma}
\label{EintensitiesENi}
Under Definition \ref{Markovianity},  which implies the Markov property of the system, and taking into account the observations of Remark 
\ref{widesensestazionarity}, we get 
$\varphi_{ij}(t) =  \alpha_{ij} e^{-\beta_i t}$, for $i,j = 1, 2$, and $\psi(t) =  \alpha_3 e^{-\beta_3 t}$. 
For $i = 1, 2$ and $i \neq j$, setting
\begin{equation}
\label{Aij-setting}
A(i,j) =  \frac{\lambda^i_0 }{\beta_i} \cdot \frac{\alpha_{ij} \alpha_{ji} -  \alpha_{ii} \alpha_{jj}}{\beta_i-\beta_j}  
-\frac{\lambda^i_0 }{\beta_i} \alpha_{ii} -  \frac{\lambda^j_0 }{\beta_i} \alpha_{ji} 
\qquad 
B(i,j) =  \frac{\lambda^i_0}{\beta_j} \cdot \frac{\alpha_{ii} \alpha_{jj} - \alpha_{ij} \alpha_{ji} }{\beta_i-\beta_j},
\end{equation}
and 
$$C(i,j) = \lambda^i_0  - A(i,j) - B(i,j),$$
by Theorem \ref{EBElambda} and Theorem \ref{Exlambda3}, the expected number of births, the expected number of deaths and the corresponding expected number of intensities reduce to
\begin{equation}
\label{E-lambda1}
\mathbb{E} [ \lambda^i(t) ] =  C(i,j)  + 
A(i,j) e^{-\beta_i t} + B(i,j) e^{-\beta_j t},
\end{equation}
$$\mathbb{E} [ N^i(t) ] = \int_0^t \mathbb{E} [ \lambda^i(s) ] \ ds  
= C(i,j) t  + 
A(i,j) \frac{1-e^{-\beta_i t}}{\beta_i} + B(i,j) \frac{1-e^{-\beta_j t}}{\beta_i},$$
\begin{eqnarray}
\label{E-lambda3}
\mathbb{E} [ \lambda^3(t) ] 
&=& \lambda^3_0 
+ \lambda^3_0 (1 - e^{- \beta_3 t}) \frac{\alpha_3}{\beta_3},
\\
\mathbb{E} [ N^3(t) ] 
&=& \lambda^3_0 \left( 1 + \frac{\alpha_3}{\beta_3} \right) t 
+ \lambda^3_0 ( e^{- \beta_3 t} - 1) \frac{\alpha_3}{(\beta_3)^2}.
\nonumber
\end{eqnarray}
\end{lemma}
\begin{proof}
Recalling that, in the general case,
\begin{eqnarray}
\mathbb{E} [ \lambda^i(t) ] &=&  \lambda^i_0  \int_0^t \int_0^s
\left[ \varphi_{ij}(s-u) \varphi_{ji}(u)
-\varphi_{ii}(s-u) 
\varphi_{jj}(u) 
\right] \ du  \ ds
\nonumber\\
&&+ \lambda^i_0 
+ \lambda^i_0  \int_0^t \varphi_{ii}(s) \ ds 
+ \lambda^j_0  \int_0^t \varphi_{ji}(s) \ ds,
\nonumber
\end{eqnarray}
substituting, successively we get that
\begin{eqnarray}
\varphi_{ij}(s-u)  \varphi_{ji}(u)
-\varphi_{ii}(s-u) 
\varphi_{jj}(u) 
&=&  
\alpha_{ij} \alpha_{ji} e^{-\beta_j s}  e^{(\beta_j-\beta_i) u}
- \alpha_{ii} \alpha_{jj} e^{-\beta_i s } e^{(\beta_i -\beta_j) u},
\nonumber
\end{eqnarray}
and
\begin{eqnarray}
&&\int_0^s \left[ \varphi_{ij}(s-u)  \varphi_{ji}(u)
-\varphi_{ii}(s-u) 
\varphi_{jj}(u)  \right] \ du 
\nonumber\\
&&\hspace{1 in} = 
\alpha_{ij} \alpha_{ji} e^{-\beta_j s}  \frac{e^{(\beta_j-\beta_i) s} - 1}{\beta_j-\beta_i}
- \alpha_{ii} \alpha_{jj} e^{-\beta_i s } \frac{e^{(\beta_i-\beta_j) s} - 1}{\beta_i-\beta_j}
\nonumber\\
&&\hspace{1 in} = 
\frac{\alpha_{ij} \alpha_{ji} -  \alpha_{ii} \alpha_{jj}}{\beta_j-\beta_i} e^{-\beta_i s} 
+ \frac{\alpha_{ii} \alpha_{jj} - \alpha_{ij} \alpha_{ji} }{\beta_j-\beta_i} e^{-\beta_j s}.
\nonumber
\end{eqnarray}
Therefore, 
\begin{eqnarray}
&&\mathbb{E} [ \lambda^i(t) ] 
=  \lambda^i_0  \int_0^t 
\left[  \frac{\alpha_{ij} \alpha_{ji} -  \alpha_{ii} \alpha_{jj}}{\beta_j-\beta_i} e^{-\beta_i s} 
+ \frac{\alpha_{ii} \alpha_{jj} - \alpha_{ij} \alpha_{ji} }{\beta_j-\beta_i} e^{-\beta_j s}
\right]   \ ds
\nonumber\\
&& \hspace{2 in}
+ \lambda^i_0  + \lambda^i_0  \int_0^t \alpha_{ii} e^{-\beta_i s}  \ ds + \lambda^j_0  \int_0^t \alpha_{ji} e^{-\beta_i s} \ ds
\nonumber\\
&&=  \lambda^i_0  + \left(\lambda^i_0  \left(\alpha_{ii} + \frac{\alpha_{ij} \alpha_{ji} -  \alpha_{ii} \alpha_{jj}}{\beta_j-\beta_i} \right) + \lambda^j_0 \alpha_{ji} \right)\frac{1-e^{-\beta_i t}}{\beta_i}  
+ \lambda^i_0 \frac{\alpha_{ii} \alpha_{jj} - \alpha_{ij} \alpha_{ji} }{\beta_j-\beta_i} \frac{1-e^{-\beta_j t}}{\beta_j}.  
\nonumber
\end{eqnarray}
\end{proof}
Note that in the Poisson case ($\alpha_{ij} = 0$ and $\alpha_3 =0$),  
$\mathbb{E} [ \lambda^i(t) ] = \lambda^i_0$, for $i=1,2,3$.
\begin{remark}
\label{stability+Markov}
The stability of the system (see Remark \ref{general-stability}) implies that
$$\mathbb{E} [ \lambda^3(t) ] > \mathbb{E} [ \lambda^1(t) ] + \mathbb{E} [ \lambda^2(t) ],$$ 
and, in the general case, this condition is given by 
(\ref{stab-cond}).
In the Markov case, under the hypothesis of Lemma \ref{EintensitiesENi}, 
and under (\ref{ws-stat-behav-0}) ($\delta_{ij}=0$, $i,j =1,2$, and $\delta_3=0$), the stability condition requires that
$$\lambda^3_0 
+ \lambda^3_0 (1 - e^{- \beta_3 t}) \frac{\alpha_3}{\beta_3} 
> C(1,2)  + C(2,1) + 
[A(1,2) + B(2,1)]e^{-\beta_1 t} + [ A(2,1) + B(1,2) ] e^{-\beta_2 t},$$
where
\begin{eqnarray}
A(1,2) + B(2,1) 
&=& \frac{\lambda^1_0 }{\beta_1} \cdot \frac{\alpha_{12} \alpha_{21} -  \alpha_{11} \alpha_{22}}{\beta_1-\beta_2}  
-\frac{\lambda^1_0 }{\beta_1} \alpha_{11} -  \frac{\lambda^2_0 }{\beta_1} \alpha_{21} 
+ \frac{\lambda^2_0}{\beta_1} \cdot \frac{\alpha_{22} \alpha_{11} - \alpha_{21} \alpha_{12} }{\beta_2-\beta_1}
\nonumber\\
&=&
 \frac{(\lambda^1_0 + \lambda^2_0 ) (\alpha_{12} \alpha_{21} - \alpha_{11} \alpha_{22} ) }{\beta_1(\beta_1-\beta_2)}  
-\frac{\lambda^1_0 \alpha_{11} }{\beta_1}  -  \frac{\lambda^2_0 \alpha_{21} }{\beta_1},
\nonumber
\end{eqnarray}
$$A(2,1) + B(1,2) = \frac{(\lambda^1_0 + \lambda^2_0 ) (\alpha_{12} \alpha_{21} - \alpha_{11} \alpha_{22} ) }{\beta_2(\beta_2-\beta_1)}  
-\frac{\lambda^2_0 \alpha_{22} }{\beta_2}  -  \frac{\lambda^1_0 \alpha_{12} }{\beta_2},$$
\begin{eqnarray}
C(1,2)  + C(2,1) 
&=& \lambda^1_0  - A(1,2) - B(1,2) + \lambda^2_0  - A(2,1) - B(2,1)
\nonumber\\
&=& \lambda^1_0  + \lambda^2_0  -  \frac{(\lambda^1_0 + \lambda^2_0 ) (\alpha_{12} \alpha_{21} - \alpha_{11} \alpha_{22} ) }{\beta_1(\beta_1-\beta_2)}  
+\frac{\lambda^1_0 \alpha_{11} + \lambda^2_0 \alpha_{21} }{\beta_1}  
\nonumber\\
&& \qquad - \frac{(\lambda^1_0 + \lambda^2_0 ) (\alpha_{12} \alpha_{21} - \alpha_{11} \alpha_{22} ) }{\beta_2(\beta_2-\beta_1)}  
+\frac{\lambda^2_0 \alpha_{22} + \lambda^1_0 \alpha_{12}  }{\beta_2}.
\nonumber
\end{eqnarray}
\end{remark}

\section{A virus-like evolving population}
\label{Rahul-model}

Let $\{\tau_n : n \geq 0\}$ be  the jump time of $(N^1, N^2, N^3)$, obtained by merging  the jump times of each $N^i$, $i=1,2,3$, where $N^1$ counts the number of births of the mutants, $N^2$ counts the number of births of the non-mutants, and $N^3$ counts the number of deaths. 
At $\tau_0=0$, there is no  individual at site $0$. Also, at $\tau_1$, there is necessarily a birth of a mutant. At the jump time $\tau_n$, $n \geq 2$, there is either a birth of a mutant, a birth of a non-mutant or a death of an individual from the existing population.

In case of a birth, there are two possibilities: 
a \textit{mutant} is born, which means that there exists a value $k$ such that $\tau_n = t^1_k$,  namely the $n$-th jump time of $N$ correspond to the $k$-th birth time of a mutant. Its fitness parameter
$f$ is a continuous uniformly random variable taking value  in $[0, 1]$. 
Otherwise, a \textit{non mutant} is born, which means that there exists a value $k$ such that $\tau_n = t^2_k$,  namely the $n$-th jump time of $N$ correspond to the $k$-th birth time of a non mutant individual. Its fitness parameter
$f$ is chosen with a probability proportional to the number of individuals with fitness $f$ among the entire population present at that time. Here we have a caveat that, if there is no individual present at the time of birth, then the fitness of the individual is sampled uniformly in $[0, 1]$.

In case of a death, an individual from the population at the site closest to $0$ is eliminated.

For a mutant, the fitness parameter cannot match, almost surely, any of the fitness parameters already existing in the population. 
Here and henceforth, a site represents a fitness level. 
The total population
at time $t$ is divided in exactly $l_t$ sites $x_1, \ldots , x_{l_t}$, with the size of the population at site $x_i$ being exactly $k^i_t$. 
To describe the  partition of the population at time $t$, let
$X_t = \{(k^i_t, x_i) : k^i_t \geq 1, x_i \in [0,1], i = 1, \ldots, l_t\}$. 
The couple $(k^i_t, x_i)$ belongs to the partition $X_t$ only if $k^i_t \geq 0$.
For example in Figure 1, there is an empty circle corresponding to the fitness parameter $x_5$, which means that all the individuals belonging to that specie have died by time $t$, 
therefore the couple $(0, x_5) \notin X_t$. In case there is no individual present at time $t$ in the entire population, we take $X_t = \emptyset$. 
Hence, $X_t$ takes values on the state space
$\mathbb{X} = \{ \emptyset \} \cup 
\left\{ \{ (k, x) \}_{x \in \Lambda} : (k, x) \in \mathbb{N} \times [0,1], \ \# \Lambda < \infty \right\}$.

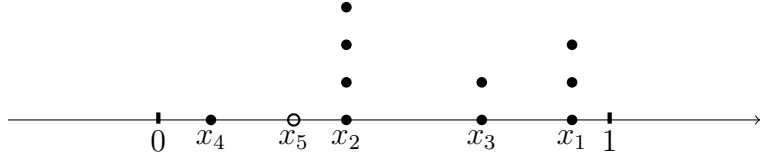
\begin{figure}
\centering
\begin{tikzpicture}
\draw [->] (-2,0) - - (8,0);
\coordinate [label=below:$0$] (0) at (0,0);
\coordinate  [label=below:$1$] (1) at (6,0);
\coordinate  [label=below:$x_1$] (p11) at (5.5,0);
\coordinate  [label=below:$x_2$] (p21) at (2.5,0);
\coordinate  [label=below:$x_3$] (p31) at (4.3,0);
\coordinate  [label=below:$x_4$] (p41) at (0.7,0);
\coordinate  [label=below:$x_5$] (p51) at (1.8,0);
\coordinate  (p12) at (5.5,0.5);
\coordinate  (p13) at (5.5,1);
\coordinate  (p22) at (2.5,0.5);
\coordinate  (p23) at (2.5,1);
\coordinate  (p24) at (2.5,1.5);
\coordinate  (p32) at (4.3,0.5);
\draw[ultra thick] (0, -0.05) -- (0,0.1);
\draw[ultra thick] (6, -0.05) -- (6,0.1);
\draw[thick] (1.8, 0) circle (0.075);
\fill (p11) circle (2pt) (p21) circle (2pt) (p31) circle (2pt) (p41) circle (2pt);
\fill (p32) circle (2pt) 
(p22) circle (2pt) (p23) circle (2pt) (p24) circle (2pt) 
(p12) circle (2pt) (p13) circle (2pt);
\end{tikzpicture}
\caption{A partition of the population at time $t$ is $X_t= \{(1,x_4), (4,x_2),(2,x_3),(3,x_1)\}$}
\label{Fig.1}
\end{figure}


For a given constant $f \in (0, 1)$, at time $t$, let $L^f(t)$ be the size of the population with fitness  smaller than $f$, and let $R^f(t)$ denote the size of the population with fitness  larger than $f$, respectively having the representations given by
\begin{equation}
\label{Lft-Rft}
L^f(t) := \sum_{i = 1}^{l_t} k^i_t \ {\bf 1}_{x_i \in [0, f]}  {\bf 1}_{(k^i_t, x_i) \in X_t},
\qquad R^f(t) := \sum_{i = 1}^{l_t} k^i_t \ {\bf 1}_{x_i \in (f, 1]}  {\bf 1}_{(k^i_t, x_i) \in X_t}.
\end{equation}
Clearly, $L^f(t) + R^f(t) := N(t)$ almost surely, where 
$N(t)$ denotes the size of the entire population at time $t$. 
The couple $\left(L^f(t), R^f(t)\right)$ takes values on $\mathbb{Z}^+ \times \mathbb{Z}^+$, 
($\mathbb{Z}^+ = \mathbb{N} \cup \{0\}$), and its dynamics is driven by 
$\left( N^1, \lambda^1, N^2,  \lambda^2, N^3, \lambda^3 \right)$. 
Hence, the model given by  $\left(L^f(t), R^f(t)\right)$ admits jumps at times $\{\tau_n\}_{n > 0}$, 
while, between the jump times, 
it remains constant, i.e.
$\left(L^f(t), R^f(t)\right) = \left(L^f(\tau_n), R^f(\tau_n)\right)$, for $\tau_n \leq t < \tau_{n+1}$.

Assuming that the structure of the Hawkes processes $\left( N^1, \lambda^1, N^2,  \lambda^2, N^3, \lambda^3 \right)$ are such that 
the model $\left(L^f(t), R^f(t)\right)$ satisfies the Markov property, (see Section \ref{MarkovSection}), at each jump time the transition probabilities are given by
\begin{enumerate}
\item
If $\left(L^f(t), R^f(t)\right) = (0,0)$
$$\left(L^f(t+h), R^f(t+h)\right) = \left\{ 
\begin{array}l
(1,0) \qquad w. p. \qquad f  \lambda^1(t)  h + o(h)
\\
(0,1) \qquad w.  p. \qquad (1-f) \lambda^1(t)  h + o(h)
\\
(0,0) \qquad w.  p. \qquad o(h)
\end{array}
\right.$$
\item 
If $\left(L^f(t), R^f(t)\right)  \in \{0\} \times \mathbb{N}$
$$\hspace{-.15 in} \left(L^f(t+h), R^f(t+h)\right)  = \left\{ 
\begin{array}l
\left(1,R^f(t)\right) \hspace{.25 in} w.  p. \  f  \lambda^1(t) h + o(h)
\\
\left(0,R^f(t) + 1 \right) 
 w.  p. \ 
[ (1-f) \lambda^1(t) + \lambda^2(t) ] h + o(h)
\\
(0,R^f(t) - 1) \ w.  p. \  \lambda^3(t) h + o(h)
\end{array}
\right.$$
\item 
If $\left(L^f(t), R^f(t)\right) \in  \mathbb{N} \times \{0\}$
$$\left(L^f(t+h), R^f(t+h)\right)  = \left\{ 
\begin{array}l
\left(L^f(t)+1, 0\right) \quad  w.  p. \quad [ f \lambda^1(t) + \lambda^2(t)] h + o(h)
\\
(L^f(t), 1) \hspace{.45 in} w.  p. \quad (1-f) \lambda^1(t) h + o(h)
\\
(L^f(t) - 1,0) \hspace{.2 in} w.  p. \quad \lambda^3(t)  h + o(h)
\end{array}
\right.$$
\item 
If $\left(L^f(t), R^f(t)\right) \in  \mathbb{N} \times \mathbb{N}$
$$\left(L^f(t+h), R^f(t+h)\right)  = \left\{ 
\begin{array}l
\left(L^f(t)+1, R^f(t)\right) \quad  w. p. 
\\
\hspace{.6 in}  
\left(f  \lambda^1(t) + \lambda^2(t) \frac{L^f(t)}{N(t)} \right) h + o(h)
\\
\left(L^f(t), R^f(t) +1\right) \quad  w.  p. 
\\
\hspace{,6 in} \left( (1-f) \lambda^1(t)
+ \lambda^2(t) \frac{R^f(t)}{N(t)} \right) h + o(h)
\\
\left(L^f(t) - 1, R^f(t)\right) 
\hspace{.15 in} w.  p. \quad \lambda^3(t) h + o(h).
\end{array}
\right.$$
\end{enumerate}

\section{Convergence results for a virus-like evolving population}
\label{model-convergences}

Some convergence results for the model introduced in Section \ref{Rahul-model} are derived below.
\begin{theorem}
\label{T1-Rahul}
Assuming that 
\begin{equation}
\label{fc}
f_c := \lim_{t \rightarrow \infty} f_c(t) 
:= \lim_{t \rightarrow \infty} \frac{\mathbb{E}[\lambda^3(t)]}{\mathbb{E}[\lambda^1(t)]} 
\end{equation}
exists, 
there is a phase transition at a critical fitness level $f_c$, 
in the sense that
\begin{enumerate}
\item
In case 
$\displaystyle \liminf_{t \rightarrow + \infty} \frac{\mathbb{E}[\lambda^3(t)]}{\mathbb{E}[\lambda^1(t) + \lambda^2(t)]} \geq 1$, 
the probability that a birth happens is less likely than a death, then
$\displaystyle P \left( \liminf_{\tau \rightarrow + \infty}  \frac{ l\left\{ t \in [0,\tau]: N(t) = 0 \right\} }{\tau}> 0 \right) = 1$, 
where $l$ is the Lebesgue measure.
\item
In case $f_c \leq 1$ $(^{*})$, \footnote{$(^{*})$ Since $\mathbb{E}[\lambda^2(t)] \geq 0$, 
the condition $f_c := \lim_{t \rightarrow \infty}\frac{\mathbb{E}[\lambda^3(t)]}{\mathbb{E}[\lambda^1(t)]} \leq 1$, obviously, implies that $\limsup_{t \rightarrow \infty}\frac{\mathbb{E}[\lambda^3(t)]}{\mathbb{E}[\lambda^1(t) + \lambda^2(t)]} \leq 1$.}
$N(t)$ goes to infinity as $t \rightarrow \infty$, 
and most of the population is distributed at sites in the interval $[f_c,1]$, in the sense that, for $f > f_c$,
$$P\left(\lim_{t \rightarrow \infty} \frac{R^{f_c}(t)}{N(t)} =1 \right) = 1
\qquad \mbox{and} \qquad P\left(\liminf_{t \rightarrow \infty} \frac{R^{f_c}(t)- R^f(t)}{N(t)} > 0 \right) =1.$$
\item
In case $\displaystyle \limsup_{t \rightarrow + \infty} \frac{\mathbb{E}[\lambda^3(t)]}{\mathbb{E}[\lambda^1(t) + \lambda^2(t)]} < 1$, 
but $f_c \geq 1$,  
$N(t)$ goes to infinity as $t \rightarrow \infty$, 
and most of the population is concentrated at site near $1$, in the sense that
$$P\left(\lim_{t \rightarrow \infty} N(t) = \infty \right) =1
\qquad  \mbox{and} \qquad  
P\left(\lim_{t \rightarrow \infty} \frac{R^{1 - \epsilon}(t)}{N(t)} =1 \right) = 1,\quad \mbox{for} \ \epsilon > 0.$$
\end{enumerate}
\end{theorem}

First we need some remarks and then some preliminaries in Lemma \ref{fc-delta0} and Lemma \ref{Lemma4-Rahul} below,  
successively,  we are able to prove Theorem \ref{T1-Rahul}. 

\begin{remark}
\label{def-of-fc}
Recall that in the model developed by 
\citet{RoyTanemura}, $f_c:= \frac{1-p}{pr}$, and a birth or a death takes place at a discrete time point with probability $p$ or $1-p$ respectively; an individual is either a mutant with probability $p \cdot r$ or an individual of one of the existing subspecies with probability $p(1-r)$.

In the present paper, 
$\displaystyle f_c := \lim_{t \rightarrow + \infty} \frac{\mathbb{E}[\lambda^3(t)]}{\mathbb{E}[\lambda^1(t)]}$, 
and a birth or a death takes place at continuous time, following a Hawkes process with intensity 
$\lambda^1(t) + \lambda^2(t)$ or $\lambda^3(t)$, respectively; at birth an individual is either a mutant with intensity $\lambda^1(t)$ or an individual of one of the existing subspecies with intensity $\lambda^2(t)$.
\end{remark}

\begin{remark}
\label{situation-to-study}
When the mean birth rate, $\mathbb{E} [\lambda^1(t) + \lambda^2(t)]$, 
is less than the mean death rate, $\mathbb{E} [\lambda^3(t)]$, 
this means that $\lim_{t \rightarrow \infty} \frac{\mathbb{E} [\lambda^3(t)]}{\mathbb{E} [\lambda^1(t) + \lambda^2(t)]} > 1$, then for $n$ big enough, 
$\{N(\tau_n) : n \geq 0\}$ is equivalent to a random walk on $\mathbb{Z}^+$ with non-positive drift and reflection at $0$. 
Thus, $N(\tau_n)$ returns to $0$ infinitely often with probability $1$.

When the mean birth rate, $\mathbb{E} [\lambda^1(t) + \lambda^2(t)]$, 
is greater than the mean death rate, $\mathbb{E} [\lambda^3(t)]$, 
$\{N(\tau_n) : n \geq 0\}$ is equivalent to a random walk on $\mathbb{Z}^+$ 
with positive drift and, 
$N(\tau_n) \rightarrow \infty$ as $n \rightarrow \infty$ with probability $1$. 

Then, we have to study the situation in which 
$\mathbb{E} [\lambda^1(t) + \lambda^2(t)] > \mathbb{E} [\lambda^3(t)]$.

\end{remark}

\begin{corollary}
\label{Cor2-Rahul}
Let 
$\displaystyle F_t(f) := \frac{\# \left\{ s \in [0,f]: (k,s) \in X_t, \ \mbox{for some} \  k \geq 1 \right\}}{\# \left\{ s \in [0,1]: (k,s) \in X_t, \ \mbox{for some} \  k \geq 1 \right\}}$
be the empirical distribution of sites at $t$. 
If $f_c < 1$, 
$\displaystyle 
F_t(f) \rightarrow \frac{\max \{f - f_c, 0\} }{1-f_c}$ uniformly a.s..
\end{corollary}
 
When the processes attain a wide-sense stationary behaviours, 
which implies the conditions given in (\ref{ws-stat-behav-0}) and (\ref{ws-stat-behav-<}),  $f_c$ assumes a closed form expression as given in 
(\ref{def-fc}) below.

\begin{lemma}
\label{fc-delta0}
Recalling 
(\ref{Aij-setting}), (\ref{E-lambda1}) and (\ref{E-lambda3}), 
and assuming the existence of $f_c$, as defined in (\ref{fc}), 
under the conditions (\ref{ws-stat-behav-0}), Remark \ref{widesensestazionarity}, namely, $\delta_{ij} = 0$, $i,j =1,2$,  and $\delta_3 = 0$, Equation (\ref{fc}) turns to be
\begin{eqnarray}
\label{def-fc}
f_c &:=&  \lim_{t \rightarrow \infty} \frac{\mathbb{E} [\lambda^3(t)]}{\mathbb{E} [\lambda^1(t)]} 
\nonumber\\
&=&  \lim_{t \rightarrow \infty} \frac{\lambda^3_0 
+ \lambda^3_0 (1 - e^{- \beta_3 t}) \frac{\alpha_3}{\beta_3}}{C(1,2)  + 
A(1,2) e^{-\beta_1 t} + B(1,2) e^{-\beta_2 t}}
\nonumber\\
&=&  \frac{\lambda^3_0 (\alpha_3 + \beta_3) \beta_1 \beta_2}{\left[\lambda^1_0 \beta_1 \beta_2 - \lambda^1_0 (\alpha_{11} \alpha_{22} - \alpha_{12} \alpha_{21}) + \beta_2 (\lambda^1_0 \alpha_{11} + \lambda^2_0 \alpha_{21} ) \right] \beta_3}.
\end{eqnarray}
Furthermore, assuming also 
the conditions (\ref{ws-stat-behav-<}) set in Remark \ref{widesensestazionarity},
namely 
$\alpha_{ji} \leq \beta_i$, $i,j =1,2$, and $\alpha_3 \leq \beta_3$, 
we have that $\displaystyle \frac{\lambda_0^3}{2\lambda_0^1+\lambda_0^2}  \leq f_c \leq 2 \frac{\lambda^3_0}{\lambda^1_0}$.
\end{lemma}

\begin{remark}
As a special case, assuming that $\alpha_{ij} = \alpha_3 = \alpha$, constant, and $\beta_i = \beta$, for all $i$, $j$, then 
$f_c = \lambda^3_0 \left(\lambda^1_0 + \lambda^2_0 \frac{\alpha}{\alpha+ \beta} \right)^{-1} \leq \ \frac{\lambda^3_0}{\lambda^1_0}$.
\end{remark}

\begin{lemma} 
\label{Lemma4-Rahul}
Under all the assumptions made so far, by the markovianity of the model,
taking into account that $f_c$ is given by 
(\ref{def-fc}), we get that
\begin{enumerate}
\item If $f_c < 1$, we have to consider two situations
\begin{enumerate}
\item For $f < f_c$ and for any 
$\eta \in (0,1)$, 
then 
\begin{eqnarray}
\label{Eq1a1}
&&P \left( \exists \ T > 0 : \rho^f_t := \frac{L^f(t)}{N(t)} \leq \eta, \forall t \geq T \right)= 1, 
\quad \mbox{and} 
\\
\label{Eq1a2}
&& \qquad P \left( L^f(\tau_n) = 0 \ \mbox{infinitely often} \right) = 1.
\end{eqnarray} 
\item For $f > f_c$, then
\begin{equation}
    \label{Eq1b} 
    P \left( L^f(\tau_n) = 0 \ \mbox{infinitely often} \right) = 0. 
\end{equation} 
\end{enumerate}
\item 
If $f_c \geq 1$, then
\begin{enumerate}
\item 
For $f < 1$ and for $\eta \in (0,1)$, since $f \in (0,1)$, we have 
(\ref{Eq1a1}) and (\ref{Eq1a2}),
\item
For $f = 1$ we have 
(\ref{Eq1b}).
\end{enumerate}

\end{enumerate}
\end{lemma}

\begin{proof}
The idea of the proof is to prove 1. and 2. together. 
To this end, note that choosing   
$f < f_c \wedge 1$, we have that $R^f(t)$ is much larger than $L^f(t)$. 
Hence, the main tool is to study a Markov chain, which is a modification of that given in Section \ref{Rahul-model} and constructed as follows. 
This is a procedure which is along the lines of that used in 
\citet{RoyTanemura} in the discrete case. 
Given $\epsilon \in [0,1]$, we construct a modified  Markov chain 
$\left( L^f_t(\epsilon), R^f_t(\epsilon) \right)$, which admits jumps at times 
$\{\tau_n\}_{n > 0}$, 
while, between the jump times, 
remains constant, i.e.
$\left( L^f_t(\epsilon), R^f_t(\epsilon) \right) = \left(L^f_{\tau_n}(\epsilon), R^f_{\tau_n}(\epsilon)\right)$ for $\tau_n \leq t < \tau_{n+1}$, 
and such that 
at each jump time the stationary transition probabilities are given by
\begin{enumerate}
\item
If $\left(L^f_t(\epsilon), R^f_t(\epsilon)\right) = (0,0)$
$$\left(L^f_{t+h}(\epsilon), R^f_{t+h}(\epsilon)\right) = \left\{ 
\begin{array}l
(1,0) \qquad w.  p. \qquad f  \left(\lambda^1(t) + \lambda^2(t) \right) h + o(h)
\\
(0,1) \qquad w.  p. \qquad (1-f) \left(\lambda^1(t) + \lambda^2(t) \right) h + o(h)
\\
(0,0) \qquad w.  p. \qquad o(h)
\end{array}
\right.$$
\item 
If $\left(L^f_t(\epsilon), R^f_t(\epsilon)\right) \in \{0\} \times \mathbb{N}$
$$\left(L^f_{t+h}(\epsilon), R^f_{t+h}(\epsilon)\right) = \left\{ 
\begin{array}l
\left(1,R^f_t(\epsilon)\right) \hspace{.2 in} w.  p. \ f  \lambda^1(t) h + o(h)
\\
\left(0,R^f_t(\epsilon) + 1\right)  w.  p. \ \left( (1-f) \lambda^1(t) + \lambda^2(t) \right) h \hspace{-.03 in} + \hspace{-.03 in} o(h)
\\
\left(0,R^f_t(\epsilon) - 1\right)  w.  p. \ \lambda^3(t) h + o(h)
\end{array}
\right.$$
\item 
If $\left(L^f_t(\epsilon), R^f_t(\epsilon) \right) \in  \mathbb{N} \times \{0\}$
$$\left(L^f_{t+h}(\epsilon), R^f_{t+h}(\epsilon)\right) = \left\{ 
\begin{array}l
\left(L^f_t(\epsilon) +1, 0 \right) \quad  w. p. \quad \left( f \lambda^1(t) + \lambda^2(t) \right) h + o(h)
\\
\left(L^f_t(\epsilon), 1 \right) \hspace{.38 in} w.  p. \quad (1-f) \lambda^1(t) h + o(h)
\\
\left(L^f_t(\epsilon) - 1,0 \right) \quad w. p. \quad \lambda^3(t) h + o(h)
\end{array}
\right.$$
\item 
If $\left(L^f_t(\epsilon), R^f_t(\epsilon)\right) \in  \mathbb{N} \times \mathbb{N}$
$$\left(L^f_{t+h}(\epsilon), R^f_{t+h}(\epsilon)\right) = \left\{ 
\begin{array}l
\left(L^f_t(\epsilon) +1, R^f_t(\epsilon) \right)   w. p. 
\left(f  \lambda^1(t) + \epsilon \lambda^2(t) \right) h + o(h)
\\
\left(L^f_t(\epsilon), R^f_t(\epsilon) +1\right)  w.  p. 
\\
\hspace{.7 in}  \left( (1-f) \lambda^1(t)
+ (1 - \epsilon) \lambda^2(t)  \right) h + o(h)
\\
\left(L^f_t(\epsilon) - 1, R^f_t(\epsilon) \right)   w. p. \ \lambda^3(t) h + o(h).
\end{array}
\right.$$
\end{enumerate}
The difference between this Markov chain and that given in Section \ref{Rahul-model} is in the definition of the step 4, since  we replace $\frac{L^f(t)}{N(t)}$ with $\epsilon$, and $\frac{R^f(t)}{N(t)} = \frac{N(t) - L^f(t)}{N(t)} = 1 - \frac{L^f(t)}{N(t)}$ with $1-\epsilon$. 
For $\epsilon, \epsilon' \in [0,1]$, we have that 
\begin{equation}
\label{L<R<}
L^f_t(\epsilon) \leq L^f_t(\epsilon')
\quad \mbox{and} \quad
R^f_t(\epsilon) \leq R^f_t(\epsilon')
\qquad \mbox{for} \ \epsilon \leq \epsilon' \quad \mbox{and} \quad  t \geq 0.
\end{equation}
Given $L^f(t), R^f(t)$ and $N(t)$ as  in  Section \ref{Rahul-model}, Equations (\ref{Lft-Rft}), 
$N(t) = L^f(t) + R^f(t)$. Taking 
$L^f_t(\epsilon), R^f_t(\epsilon)$ and $N_t(\epsilon)$ as above, and choosing
$\rho^f_t := \frac{L^f(t)}{N(t)}$, we get
\begin{eqnarray}
\label{L0<L1,R1<R0}
&L^f(t+h) = L^f_{t+h}(\rho^f_t), \qquad  R^f(t+h) = R^f_{t+h}(\rho^f_t),&
    \\
   &N_t(\epsilon) := L^f_t(\epsilon) + R^f_t(\epsilon) = N(t), \quad L^f_t(0) \leq L^f_t  \leq L^f_t(1), \quad 
    R^f_t(1) \leq R^f_t  \leq R^f_t(0).&
    \nonumber
\end{eqnarray}
Setting, as usual, $\Delta L^f_t(\epsilon) = L^f_t(\epsilon) - L^f_{t-}(\epsilon)$ and 
$\Delta N_t(\epsilon) = N_t(\epsilon) - N_{t-}(\epsilon)$, 
and recalling that the number of jumps of the chain until time $t$ is
$N^1_t(\epsilon) + N^2_t(\epsilon) + N^3_t(\epsilon)$, 
by the Law of Large Numbers, we get that
$$\lim_{t \rightarrow \infty} \frac{L^f_t(\epsilon)}{N^1_t(\epsilon) + N^2_t(\epsilon) + N^3_t(\epsilon)} 
= \lim_{t \rightarrow \infty} \mathbb{E}\left[ \Delta L^f_t(\epsilon) \right]
\hspace{3 in}$$
$$= f \lambda^1_0 + \epsilon \lambda^2_0 - \lambda^3_0
+ f  \left( \frac{\lambda^1_0 \alpha_{11} }{\beta_1} + \frac{\lambda^2_0 \alpha_{21} }{\beta_1} \right)
+ \epsilon \left( \frac{\lambda^2_0 \alpha_{22} }{\beta_2} + \frac{\lambda^1_0  \alpha_{12} }{\beta_2} \right)
- \lambda^3_0 \frac{\alpha_3}{\beta_3}
- (f \lambda^1_0 + \epsilon \lambda^2_0 ) \frac{\alpha_{11} \alpha_{22} - \alpha_{12} \alpha_{21}}{\beta_1 \beta_2},
$$
and
$$\lim_{t \rightarrow \infty} \frac{N_t(\epsilon)}{N^1_t(\epsilon) + N^2_t(\epsilon) + N^3_t(\epsilon)} 
= \lim_{t \rightarrow \infty} \mathbb{E}\left[ \Delta N_t(\epsilon) \right]
\hspace{3 in}$$
$$= \lambda^1_0 + \lambda^2_0 - \lambda^3_0
+ \left( \frac{ \lambda^1_0 \alpha_{11} }{\beta_1} + \frac{ \lambda^2_0 \alpha_{21} }{\beta_1} \right)
+  \left( \frac{\lambda^2_0 \alpha_{22} }{\beta_2} + \frac{\lambda^1_0 \alpha_{12} }{\beta_2} \right)
- \lambda^3_0 \frac{\alpha_3}{\beta_3} 
- (\lambda^1_0 + \lambda^2_0 ) \frac{\alpha_{11} \alpha_{22} - \alpha_{12} \alpha_{21}}{\beta_1 \beta_2},
$$
almost surely. Thus, setting 
$\Lambda_i := \lim_{t \rightarrow \infty} \mathbb{E}[\lambda^i(t)]$, for $i=1,2,3$, and 
$\rho^f_t(\epsilon) := \frac{L^f_t(\epsilon)}{N_t(\epsilon)}$, 
\begin{eqnarray}
\label{rhotf}
\lim_{t \rightarrow \infty} \rho^f_t(\epsilon)  
&:=& \lim_{t \rightarrow \infty} \frac{L^f_t(\epsilon)}{N_t(\epsilon)} 
=  \lim_{t \rightarrow \infty} \frac{f  \mathbb{E}[\lambda^1(t)] + \epsilon \mathbb{E}[\lambda^2(t)] - \mathbb{E}[\lambda^3(t)]}{\mathbb{E}[\lambda^1(t)] + \mathbb{E}[\lambda^2(t)] - \mathbb{E}[\lambda^3(t)]} 
\nonumber\\
&=&  
\frac{f  \Lambda_1 - \Lambda_3}{\Lambda_1 + \Lambda_2 - \Lambda_3}
+ 
\frac{ 
\epsilon \Lambda_2}{\Lambda_1 + \Lambda_2 - \Lambda_3}.
\end{eqnarray}
Let us introduce the linear function
$h(x) := \frac{f  \Lambda_1 - \Lambda_3}{\Lambda_1 + \Lambda_2  - \Lambda_3 } + \frac{\Lambda_2}{\Lambda_1 + \Lambda_2  - \Lambda_3 } \ x.$
Since, for $i,j=1,2$, and $j\neq i$,
$\Lambda_i =  \lambda^i_0 
+ 
\left(\frac{\lambda^i_0 \alpha_{ii} }{\beta_i}
+ \frac{\lambda^j_0 \alpha_{ji} }{\beta_j}
 \right) 
- \lambda^i_0 \frac{\alpha_{ii} \alpha_{jj} - \alpha_{ij} \alpha_{ji}}{\beta_i \beta_j}$, and 
$\Lambda_3 = \lambda^3_0 + \lambda^3_0 \frac{\alpha_3}{\beta_3}$, 
without loss of generality, for the present model we assume that $\Lambda_i > 0$, for $i=1,2,3$, and by the stability condition in the Markov case, we have that $\Lambda_1+\Lambda_2 - \Lambda_3 > 0$. Hence $h(x)$ is an increasing line, and recalling that $f \in(0,1)$, for $\Lambda_1 < \Lambda_3$,
$$h(0) = \frac{ f \Lambda_1 - \Lambda_3}{\Lambda_1 + \Lambda_2  - \Lambda_3 } < 0, \quad 
h(1) := \frac{f  \Lambda_1 - \Lambda_3}{\Lambda_1 + \Lambda_2  - \Lambda_3 } + \frac{\Lambda_2}{\Lambda_1 + \Lambda_2  - \Lambda_3 } < 1.$$
Thus, there exists a real $\delta>0$ such that $h(x+\delta) < x$, for all $x \in[0, 1]$. 
Setting
$H(\epsilon, \delta) := \left\{ \omega: \exists T = T(\omega) : \forall t \geq T, \rho^f_t(\epsilon) < h(\epsilon + \delta) \right\}$, 
by 
(\ref{rhotf}),  for all $\epsilon, \delta \in (0,1]$,
$\displaystyle \lim_{t \rightarrow \infty} \rho^f_t(\epsilon) = h(\epsilon) < h(\epsilon + \delta)$, 
which implies that
\begin{equation}
\label{PLambda=1}
P \left( H(\epsilon, \delta) \right) =  1.
\end{equation}
Thus, taking $\epsilon_c > 0$, such that $h(\epsilon_c + \delta) = \eta$, since $h(x)$ is an increasing line, for $\epsilon < \epsilon_c$, $h(\epsilon + \delta) < h(\epsilon_c + \delta) = \eta$. But the sequence defined by 
$x_{n+1} := h(x_n + \delta)$, is a decreasing sequence, since 
$x_{n+1} = h(x_n + \delta) < x_n$, for $x_n \in [0, 1]$. This means that
$\lim_{n \rightarrow \infty} x_n  
= h(0) = \frac{ f \Lambda_1 - \Lambda_3}{\Lambda_1 + \Lambda_2  - \Lambda_3 } < 0$. 
Setting
$h(k, x + \delta) := h \left(2^{-k} ( [2^k x] + 1 ) + \delta \right)$, 
for $k \in \mathbb{N}$, big enough, there exists $n_c \in \mathbb{N}$ such that the sequence
\begin{equation}
\label{h^n:=}
h_n = h_n(k, 1 + \delta) := h(k, h_{n-1} + \delta) \leq \eta, 
\qquad \forall n \geq n_c.  
\end{equation}
Taking into account that
by 
(\ref{L<R<}) and (\ref{L0<L1,R1<R0}), 
$\rho^f_t(\epsilon) \geq \rho^f_t(\epsilon') \quad \mbox{for} \ \epsilon > \epsilon'$, and $\rho^f_t \leq \rho^f_t(1)$, 
for all $\omega \in \cap_{m \in \mathbb{N}} H(m 2^{-k}, \delta)$, there exists $T_1(\omega)$, such that, 
$\rho^f_t [\omega] \leq \rho^f_t(1) [\omega] \leq h(k, 1+ \delta)$, for all $t \geq T_1(\omega)$, 
and there exists $T_2(\omega) \geq T_1(\omega)$, such that, for all $t \geq T_2(\omega)$, 
$\rho^f_t [\omega] \leq \rho^f_t \left(h(k, 1+ \delta) \right) [\omega] 
\leq h \left(k, h(k, 1+\delta)+ \delta \right) 
= h_2(k, 1+ \delta)$. 
Repeating this procedure, we get that for any $l \in \mathbb{N}$, there exists $T_l(\omega)$ such that, for all $t \geq T_l(\omega)$
$\rho^f_t [\omega] \leq h_l(k, 1+ \delta) \qquad \forall \ t \geq T_l(\omega)$.
From 
(\ref{h^n:=}), we have that
$\rho^f_t [\omega] \leq \eta$, for all $t \geq T_{n_c}(\omega)$.
Since $P\left( \cap_{m \in \mathbb{N}} H(m 2^{-k}, \delta) \right) = 1$, from 
(\ref{PLambda=1}), we have that 
$\displaystyle \lim_{t \rightarrow \infty} \rho^f_t [\omega] \leq \eta$, a.s., 
and the condition given in 
(\ref{Eq1a1}) is proven. 
For the condition 
(\ref{Eq1a2}), $L^f_t$ is the continuous time interpolation of  $L^f_{\tau_n}$, and that $L^f_{\tau_n}$ is recurrent for 
$$ f < f_c - \frac{\Lambda_2}{\Lambda_1} \epsilon.$$
Hence, for $f < f_c \wedge 1$, this condition holds for $\epsilon$ small enough, which implies that $L^f_{\tau_n}=0$ infinitely often with probability 1. 

To prove 
(\ref{Eq1b}),  let $f_c <1$, and note that, at $t$, the number of sites in $[0,f]$, is less than $L^f_t(0)$, and that when $f > f_c$, the number of sites in $[0,f]$ cannot be null with positive probability. This complete the proof of part 1 point b) of Lemma \ref{Lemma4-Rahul}. 
The proof of part 2 point b) of Lemma \ref{Lemma4-Rahul} follows since, for 
$1 \leq f_c < \frac{\Lambda_3}{\Lambda_1}$,  the observation given in  Remark \ref{stability+Markov} assures that $N_t \rightarrow \infty$ as $t \rightarrow \infty$ with probability $1$.
\end{proof}

Summing up, for the proof of Theorem \ref{T1-Rahul}, first, taking into account the closed form expressions obtained for $\mathbb{E}[\lambda^i(t)]$, $i=1,2,3$, under the Markov property and the wide-sense stationary behaviours,  we have that: 
Theorem \ref{T1-Rahul}.1 
    is obtained  analogously to queueing theory with mean arrival rates less than mean service rates. 

For Theorem \ref{T1-Rahul}.2, the first statement  is derived by Lemma \ref{Lemma4-Rahul}.1.b and Lemma \ref{Lemma4-Rahul}.2.b., while the second statement is derived taking into account that, considering $\lambda^1(t)$, the birth rate of mutants, the limiting expected number of them with a fitness between $(a, b)$, with $f_c < a < b \leq 1$, is $\lambda^1(t) (b-a)$. Thus, by the Strong Law of Large Numbers, 
$$\liminf_{t \rightarrow \infty}  \frac{R^a(t) - R^b(t)}{N(t)} \geq \frac{b-a}{1-f_c}
      \qquad \mbox{a.s.}.$$
Theorem \ref{T1-Rahul}.3 is derived by Lemma \ref{Lemma4-Rahul}.2.  
Finally, since the sites are independently and uniformly distributed in $[0,1]$, 
Corollary \ref{Cor2-Rahul} follows from Lemma \ref{Lemma4-Rahul}.

\appendix
\section{Proofs}
\subsection{Proof of Theorem \ref{EBElambda}}
\label{ProofThENElambda}

\begin{remark}
Recall that, any non-negative $\mathcal{F}_t$-progressively measurable process $\lambda(t)$ is called $\mathcal{F}_t$-intensity of a counting process $N(t)$, provided that, for all $s$, $t$, with $0 < s < t$, 
$\mathbb{E} \left[ N(s,t] |  \mathcal{F}_s \right] = 
\mathbb{E} \left[ N(t) - N(s)  |  \mathcal{F}_s \right] = 
\mathbb{E} \left[ \int_s^t \lambda(u) du | \mathcal{F}_s \right]$. 
This implies that, for $N(0)=0$,
$\mathbb{E} \left[ N(t) \right] =  \mathbb{E} \left[ \int_0^t \lambda(u) du \right] 
= \int_0^t \mathbb{E} \left[ \lambda(u) \right]  du$,
and that $d \mathbb{E} \left[ N(t) \right] = \mathbb{E} \left[ \lambda(t)  \right] dt$, since
\begin{eqnarray}
    d \mathbb{E} \left[ N(t) \right] 
    &=& \lim_{h \rightarrow 0} \frac{1}{h} \mathbb{E} \left[ N(t+h) - N(t) \right] dt
   = \lim_{h \rightarrow 0} \frac{1}{h} \mathbb{E} \left[ \int_t^{t+h} \lambda(u) du \right] dt
    \nonumber\\
    &=& \lim_{h \rightarrow 0} \frac{1}{h} \int_t^{t+h} \mathbb{E} \left[ \lambda(u)  \right] du \ dt
    = \lim_{h \rightarrow 0} \mathbb{E} \left[ \lambda(t+o(h))  \right] dt 
    = \mathbb{E} \left[ \lambda(t)  \right] dt.
    \nonumber
\end{eqnarray}
\end{remark}

\begin{proof} of Theorem \ref{EBElambda}. \quad 
Let $y^i_t = \mathbb{E} [ \lambda^i(t) ]$, with 
$y^i_0 = \mathbb{E} [ \lambda^i(0)] = \lambda^i_0$. 
By the ODE given in (\ref{ODExLambdaB}) we have
$(y^i_t)'  =
\sum_{j=1,2} \left( \int_0^t \varphi'_{ji}(t-u) y^j_u du +  \varphi_{ji}(0)  y^j_t \right)$. 
Taking the Laplace Transform on both sides, and setting as usual  
$Y^i(s) = {\cal L}(y^i_t)(s) = \int_0^\infty e^{-s t} y^i_t dt$,
for the LHS (left hand side) we get that
${\cal L}((y^i_t)'  )(s) = s Y^i(s) - y^i_0 = s Y^i(s) - \lambda^i_0$.
For the RHS (right hand side), we have successively
\begin{eqnarray}
&&{\cal L}\left(\sum_{j=1,2} \left( \int_0^t \varphi'_{ji}(t-u) y^j_u du +  \varphi_{ji}(0)  y^j_t \right) \right)(s)
\nonumber\\
&& = \sum_{j=1,2}  \left[ \left( s {\cal L} (\varphi_{ji}(t) )(s) - \varphi_{ji}(0)   \right) Y^j(s)
+  \varphi_{ji}(0) Y^j(s) \right] 
= s \sum_{j=1,2}  {\cal L} (\varphi_{ji}(t) )(s) Y^j(s).
\nonumber
\end{eqnarray}
Hence, we get the system
$$\left\{
\begin{array}l
s Y^1(s) - \lambda^1_0 = s {\cal L} (\varphi_{11}(t) )(s) Y^1(s) + s {\cal L} (\varphi_{21}(t) )(s) Y^2(s)
\\
s Y^2(s) - \lambda^2_0 = s {\cal L} (\varphi_{12}(t) )(s) Y^1(s) + s {\cal L} (\varphi_{22}(t) )(s) Y^2(s)
\end{array}
\right.$$
and by Cramer's rule, immediately we get, 
for $i=1,2$, $j=1,2$ and $j \neq i$,
\begin{equation*}
Y^i(s) = \frac{1}{s} \cdot
\frac{\lambda^i_0  [1-  {\cal L} (\varphi_{jj}(t) )(s)] 
+ \lambda^j_0 {\cal L} (\varphi_{ji}(t) )(s) }{[1 - {\cal L} (\varphi_{ii}(t) )(s)] [1-  {\cal L} (\varphi_{jj}(t) )(s)]
-  {\cal L} (\varphi_{ij}(t) )(s) {\cal L} (\varphi_{ji}(t) ) (s)}, 
\end{equation*}
which implies that
\begin{eqnarray}
&&
Y^i(s)  \left\{ 
1 - {\cal L} (\varphi_{ii}(t) )(s)
-  {\cal L} (\varphi_{jj}(t) )(s)
+ {\cal L} (\varphi_{ii}(t) )(s)
{\cal L} (\varphi_{jj}(t) )(s) \right.
\nonumber\\
&&\hspace{2.5 in }\left. -  {\cal L} (\varphi_{ij}(t) )(s) {\cal L} (\varphi_{ji}(t) )(s) \right\}
\nonumber\\
&&\qquad = \lambda^i_0 \frac{1}{s}
-  \lambda^i_0 \frac{1}{s} {\cal L} (\varphi_{jj}(t) )(s)
+ \lambda^j_0 \frac{1}{s} {\cal L} (\varphi_{ji}(t) )(s).
\nonumber
\end{eqnarray}
Recalling that, as well known, 
${\cal L}( f(t))(s) {\cal L} (g(t) )(s) = {\cal L} \left(\int_0^t f(t-u) g(u) \ du \right)(s)$, 
and that ${\cal L} (1)(s) = 1/s$, 
taking the inverse Laplace Transform of LHS,  we get
\begin{eqnarray}
{\cal L}^{-1}(LHS) 
&=& 
y^i_t - 
\int_0^t y^i_{t-v} \Big[
\varphi_{ii}(v) 
+ \varphi_{jj}(v) 
\nonumber\\
&&
\qquad \qquad \qquad 
+ \int_0^v 
\left[ \varphi_{ij}(v-u) ) \varphi_{ji}(u)
-\varphi_{ii}(v-u) 
\varphi_{jj}(u) 
\right] \ du \Big] \ dv.
\nonumber
\end{eqnarray}
While, taking the inverse Laplace Transform of RHS,  we get
\begin{eqnarray}
{\cal L}^{-1}(RHS) 
&=&
\lambda^i_0
-  \lambda^i_0 \int_0^t \varphi_{jj}(v) \ dv
+ \lambda^j_0 \int_0^t \varphi_{ji}(v) \ dv.
\nonumber
\end{eqnarray}
Hence,
\begin{eqnarray}
y^i_t &=& 
\int_0^t y^i_{t-v} \left[
\varphi_{ii}(v) 
+ \varphi_{jj}(v) 
+ \int_0^v 
\left[ \varphi_{ij}(v-u) ) \varphi_{ji}(u)
-\varphi_{ii}(v-u) 
\varphi_{jj}(u) 
\right] \ du \right] dv 
\nonumber\\
&& \qquad
+ \lambda^i_0 -  \lambda^i_0 \int_0^t \varphi_{jj}(v) dv
+ \lambda^j_0 \int_0^t \varphi_{ji}(v) \ dv,
\nonumber
\end{eqnarray}
by taking the derivative of both sides and given the initial condition
$y^i_0 = \lambda^i_0$
\begin{eqnarray}
(y^i_t)'
&=&
\lambda^i_0  \left[
\varphi_{ii}(t) 
+ \int_0^t 
\left[ \varphi_{ij}(t-u) ) \varphi_{ji}(u)
-\varphi_{ii}(t-u) 
\varphi_{jj}(u) 
\right] \ du \right] 
 + \lambda^j_0  \varphi_{ji}(t).
\nonumber
\end{eqnarray}
By a standard procedure to solve this Cauchy Problem, we get the required result.
\end{proof}

\subsection{
Proof of Theorem \ref{Exlambda3}}

\begin{proof}
Let $y_t = \mathbb{E}[\lambda^3(t)]$ and, 
$y_0 = \mathbb{E}[\lambda^3(0)] 
= \lambda^3_0$. 
Then  (\ref{ODExLambdaD}) becomes
$$y_t' = \int_0^t \psi'(t-u) y_u du 
+ \psi(0)  y_t.$$ 
Analogously to the proof of Theorem \ref{EBElambda}, 
taking the Laplace Transform on both sides, and setting as usual  
$Y(s) = {\cal L}(y_t)(s) = \int_0^\infty e^{-s t} y_t dt,$
for the LHS  we get 
${\cal L}((y_t)'  )(s) = s Y(s) - y_0$, 
and for the RHS  we get
$${\cal L}\left( \int_0^t \psi'(t-u) y_u du + \psi(0)  y_t \right)(s) 
= s {\cal L} (\psi(t) )(s) Y(s).$$ 
Therefore, 
$\displaystyle Y(s) \left(1 - {\cal L} (\psi(t) )(s) \right)  = y_0 \frac{1}{s}$ 
and 
taking the inverse Laplace Transform,  
we get that 
$y_t  
- {\cal L}^{-1}\left( {\cal L} \left( \int_0^t y_{t-v} \psi(v) \ dv \right)  \right) 
= y_0$. 
Hence,
$y_t  = \int_0^t y_{t-v} \psi(v) \ dv + y_0$, and 
after differentiating both sides, we get 
$\displaystyle d \mathbb{E} [ \lambda^3(t) ] = (y_t)' = y_0 \psi(t)$.
\end{proof}

\section*{Acknowledgments}
This work was partially supported by the Gruppo Nazionale per l'Analisi Matematica la Probabilit\'a e le loro Applicazioni (GNAMPA-INdAM).

\section*{Funding} Open access funding provided by University of L'Aquila within the
CRUI-CARE Agreement.

\end{document}